\documentclass[12pt]{amsart}
\usepackage{}

\usepackage{amsmath}
\usepackage{amsfonts}
\usepackage{amssymb}
\usepackage[all]{xy}           

\usepackage{bbding}
\usepackage{txfonts}
\usepackage{amscd}
\usepackage{mathrsfs}

\usepackage[shortlabels]{enumitem}
\usepackage{ifpdf}
\ifpdf
  \usepackage[colorlinks,final,backref=page,hyperindex]{hyperref}
\else
  \usepackage[colorlinks,final,backref=page,hyperindex]{hyperref}
\fi
\usepackage{tikz}
\usepackage[active]{srcltx}

\makeatletter

\newtheorem{thm}{Theorem}[section]

\newtheorem{cor}[thm]{Corollary}
\newtheorem{pro}[thm]{Proposition}
\newtheorem{ex}[thm]{Example}
\newtheorem{rmk}[thm]{Remark}
\newtheorem{defi}[thm]{Definition}

\newcommand {\emptycomment}[1]{}

\newcommand{\be }{\begin{equation}}
\newcommand{\ee }{\end{equation}}

\newcommand{\Comp}{\mathbb C}

\newcommand{\huaA}{\mathcal{A}}

\newcommand{\br}[1]{   [ \cdot,    \cdot  ]   }

\newcommand{\Nat}{\mathbb N}

\begin{document}

\title[Pre-vertex algebras, pre-Poisson vertex algebras and their deformation quantizations ]{Pre-vertex algebras, pre-Poisson vertex algebras and their deformation quantizations}

\author{Jiabao Xu}
\address{School of Mathematics and Statistics, Northeast Normal University, Changchun 130024, China}
\email{xujb624@nenu.edu.cn}

\author{Jiefeng Liu}
\address{School of Mathematics and Statistics, Northeast Normal University, Changchun 130024, China}
\email{liujf534@nenu.edu.cn}
\vspace{-5mm}


\begin{abstract}
In this paper, we introduce pre-vertex algebras and pre-Poisson vertex algebras as 
$\mathfrak{preLie}$-algebras in the chiral and classical pseudo-tensor categories, respectively. We show that a pre-vertex algebra gives rise to a vertex algebra, that every Rota–Baxter operator on a vertex algebra induces a pre-vertex algebra,  and that  every pre-vertex algebra can be embedded into a vertex algebra equipped with  a Rota–Baxter operator. Moreover, we prove that pre-vertex algebras are equivalent to dendriform vertex algebras.  In  the Poisson setting, we demonstrate that pre-Poisson vertex algebras are obtained from Rota–Baxter operators on Poisson vertex algebras, from filtrations of pre-vertex algebras, and as classical limits of pre-vertex formal deformations of differential Zinbiel algebras. This extends the classical relationships among Rota–Baxter operators, pre-Lie algebras, and Poisson algebras to the framework of vertex algebras.
\end{abstract}


\maketitle



\allowdisplaybreaks


\section{Introduction}\label{sec:intr}
The notion of a vertex algebra was first axiomatized by R. Borcherds in the theory of monstrous moonshine \cite{Bor}, which provides a rigorous mathematical definition of the chiral part of a two-dimensional quantum field theory \cite{BPZ}. It has since developed in conjunction with string theory and infinite-dimensional Lie algebras, with applications to the geometric Langlands program. For more details on vertex algebras and their applications, see \cite{FB,FLM,Kac}. In \cite{Bakalov}, Bakalov and Kac gave an equivalent description of vertex algebras in terms of Lie conformal algebras and differential pre-Lie algebras. The notion of a Lie conformal algebra (also called a vertex Lie algebra in \cite{Pri}) was introduced by Kac \cite{Kac} as an algebraic language to encode the properties of operator product expansions in conformal field theory. Pre-Lie algebras (or left-symmetric algebras) are a class of nonassociative algebras arising from the study of convex homogeneous cones, affine manifolds, affine structures on Lie groups, and deformations of associative algebras; they have since appeared in many fields of mathematics and mathematical physics. See the survey paper \cite{Bai21} and the references therein for more details. In their seminal book \cite{BD}, Beilinson and Drinfeld generalized the notion of a vertex algebra by introducing chiral algebras in the language of $\mathfrak{D}$-modules on an arbitrary smooth algebraic curve; a vertex algebra then becomes a weakly translation-covariant chiral algebra on the affine line. In \cite{BDHK,DK06}, using the $\lambda$-bracket description of vertex algebras, the authors translate Beilinson and Drinfeld's construction of the chiral operad into the purely algebraic language of vertex algebras. In this setting, a vertex algebra is a $\mathfrak{Lie}$-algebra in the chiral pseudo-tensor category $\mathcal{C}^{cl}$. Consequently, the general construction of a cohomology complex associated to a linear operad yields the vertex algebra cohomology complex. For other versions of cohomology theories of vertex algebras, see \cite{Huang1,Huang2}.

Poisson vertex algebras, also known as Coisson algebras in the language of \cite{BD}, arise as the classical limit of vertex algebras \cite{FB} and provide an algebraic framework for the theory of Hamiltonian evolutionary partial differential equations \cite{BSK}. They generalize the notion of Poisson algebras to the setting of Lie conformal algebras, encoding the local Poisson bracket of local functionals via the Poisson vertex algebra $\lambda$-bracket. Consequently, they have become central objects in the algebraic theory of integrable systems and the formal calculus of variations. In \cite{BDHK}, the authors give an operadic description of Poisson vertex algebras and show that a Poisson vertex algebra is precisely a $\mathfrak{Lie}$-algebra in the classical pseudo-tensor category $\mathcal{P}^{cl}$. Cohomology theories for Poisson vertex algebras were systematically developed in \cite{BDHKV,BDK,BDK2,DK13,DK11b}. In \cite{LHS}, Li introduced the notion of a good increasing filtration associated to a filtered vertex algebra and proved that the associated graded vector space of a filtered vertex algebra naturally carries a Poisson vertex algebra structure. In particular, the universal enveloping vertex algebra $U(L)$ of a Lie conformal algebra $L$ admits a Poisson vertex algebra structure.

Rota–Baxter operators on vertex algebras, formulated in terms of state-field correspondences, were introduced by Bai, Guo, Liu, and Wang in \cite{BGL2} as the integral counterpart of derivations on vertex algebras and they give rise to solutions of the classical Yang–Baxter equations on vertex algebras \cite{BGL1}. The concept of Rota–Baxter operators on associative algebras was first introduced in the 1960s by Baxter \cite{Bax} and Rota \cite{Rot} in their study of fluctuation theory in probability and combinatorics. For Lie algebras, these operators were rediscovered in the 1980s as operator forms of the classical Yang–Baxter equation and the modified classical Yang–Baxter equation. Furthermore, Rota–Baxter operators on operads induce splittings of operads \cite{BBGN,PBG}, which generalize the relationships from Rota–Baxter algebras to dendriform algebras \cite{Agu01} and from Rota–Baxter Lie algebras to pre-Lie algebras \cite{AB}.   See \cite{GK13} for more details on the embedding of dendriform type algebras into  Rota-Baxter algebras.  Generalizing the connection between Rota–Baxter algebras and dendriform algebras, the authors in  \cite{BGL2} introduced the notion of dendriform vertex algebras, which can be viewed as the splitting of Rota–Baxter actions on vertex algebras. On the other hand, Bakalov and Kac used Lie conformal algebras and pre-Lie algebras to give an equivalent definition of vertex algebras, which can be interpreted as a $\mathfrak{Lie}$-algebra in the chiral pseudo-tensor category $\mathcal{C}^{cl}$. It is therefore natural to ask whether Rota–Baxter operators can be defined on vertex algebras in the sense of Bakalov and Kac's definition, what the corresponding notion of dendriform vertex algebra would be in this setting, and whether such algebras can be interpreted as a $\mathfrak{preLie}$-algebra in the chiral pseudo-tensor category $\mathcal{C}^{cl}$.

The first main goal of this paper is to answer these questions. We define Rota–Baxter operators on Bakalov–Kac type vertex algebras and introduce the notion of a pre-vertex algebra, which consists of a pre-Lie conformal algebra together with a differential L-dendriform algebra. Recall that pre-Lie conformal algebras (also called left-symmetric conformal algebras) were introduced by Hong and Li \cite{HL} as a conformal analogue of pre-Lie algebras; they are useful for constructing new vertex algebras. On the other hand, L-dendriform algebras were introduced by Bai, Liu and Ni \cite{BLN} as a splitting of Rota–Baxter actions on pre-Lie algebras. Our main results are as follows. We show that every pre-vertex algebra gives rise to a vertex algebra, and conversely, a Rota–Baxter operator on a vertex algebra induces a pre-vertex algebra. Moreover, every pre-vertex algebra can be embedded into a vertex algebra with a Rota–Baxter operator. In particular, any differential Zinbiel algebra is a pre-vertex algebra. This relationship generalizes the classical connection between Rota–Baxter operators on Lie algebras and pre-Lie algebras. Furthermore, although dendriform vertex algebras and pre-vertex algebras arise from different  perspectives and appear distinct, we prove that they are in fact equivalent. This equivalence extends the well‑known equivalence between the two definitions of vertex algebras – one in terms of 
$\lambda$-brackets and the other in terms of the state‑field correspondence. In this proof, the Rota–Baxter operator on vertex algebras plays a crucial role. Recently, as noted by Kolesnikov in their paper \cite{Kol}, our definition of a pre-vertex algebra coincides precisely with that of a $\mathfrak{preLie}$-algebra in the chiral pseudo tensor category $\mathcal{C}^{cl}$. From a categorical perspective, the pre-vertex algebra thus arises naturally and therefore deserves closer attention.

The second main goal of this paper is to study pre-Poisson vertex algebras and their relationship with pre-vertex algebras. A pre-Poisson vertex algebra consists of a pre-Lie conformal algebra together with a differential Zinbiel algebra. Such an object can be viewed as a  $\mathfrak{preLie}$-algebra  in the classical pseudo-tensor category, and it generalizes the notion of a pre-Poisson algebra introduced by Aguiar \cite{Aguiar}. The significance of pre-Poisson vertex algebras lies in the fact that they naturally give rise to Poisson vertex algebras. We also introduce the notion of a Rota–Baxter operator on a Poisson vertex algebra and show that such an operator induces a pre-Poisson vertex algebra. Moreover, any pre-Poisson vertex algebra can be embedded into a Poisson vertex algebra equipped with a Rota–Baxter operator. In addition, just as a filtration of a vertex algebra gives rise to a Poisson vertex algebra, we introduce the notion of a filtration of a pre-vertex algebra and prove that it also yields a pre-Poisson vertex algebra. Finally, we introduce the concept of pre-vertex formal deformations of a differential Zinbiel algebra and demonstrate that pre-Poisson vertex algebras arise as the corresponding classical limits. This extends the well-known fact that a Poisson vertex algebra is the classical limit of a vertex algebra

The paper is organized as follows. Section \ref{sec:pre-vertex} sets the stage by recalling the two equivalent definitions of vertex algebras (via 
$\lambda$-brackets and via state-field correspondence). We then introduce pre-vertex algebras and Rota–Baxter operators on Bakalov–Kac type vertex algebras, and analyze the interconnections among pre-vertex algebras, vertex algebras, and Rota–Baxter operators. The section concludes by establishing the equivalence between pre-vertex algebras and dendriform vertex algebras. Section \ref{sec:pre-Poisson} first recalls the notion of a Poisson vertex algebra, then introduces pre-Poisson vertex algebras and Rota–Baxter operators on Poisson vertex algebras. We show that every pre-Poisson vertex algebra yields a Poisson vertex algebra, and conversely, a Rota–Baxter operator on a Poisson vertex algebra induces a pre-Poisson vertex algebra. Section \ref{sec:def} introduces filtrations of pre-vertex algebras and proves that such a filtration gives rise to a pre-Poisson vertex algebra. Finally, we define pre-vertex formal deformations of a differential Zinbiel algebra and demonstrate that the classical limit of such a deformation is a pre-Poisson vertex algebra.

\section{Pre-vertex algebras and dendriform vertex algebras}\label{sec:pre-vertex}
In this section, we first recall the two equivalent definitions of vertex algebras: one in terms of $\lambda$-brackets and the other via state-field correspondences. We then introduce the notion of a pre-vertex algebra, which corresponds to the pre-Lie chiral algebra introduced by Kolesnikov. We show that pre-vertex algebras can be induced by Rota–Baxter operators on vertex algebras. Finally, we compare pre-vertex algebras with dendriform vertex algebras and prove that they are equivalent, thereby generalizing the equivalence between the two aforementioned definitions of vertex algebras.
\subsection{Vertex aglebras and pre-vertex aglebras}
Using the language of $\lambda$-brackets, Bakalov and Kac \cite{Bakalov} give an equivalent description of vertex algebras in terms of Lie conformal algebras and differential pre-Lie algebras.

Recall \cite{Kac} that a {\bf Lie conformal algebra} $L$ is an $\mathbb{C}[\partial]$-module
endowed with a $\lambda$-bracket $[\cdot\,_\lambda\,\cdot]:\,L\otimes L\to L[\lambda]$
satisfying: for $x,y,z\in L$, 
\begin{itemize}
	\item[(i)] sesquilinearity:
	$[\partial x_\lambda y]=-\lambda[y_\lambda x]$,
	$[x_\lambda \partial y]=(\partial+\lambda)[x_\lambda y]$;
	\item[(ii)] skewsymmetry:
	$[y_\lambda x]=-[x_{-\lambda-\partial} y]$;
	\item[(iii)] Jacobi identity:
	$[x_\lambda[y_\mu z]]-[y_\mu[x_\lambda z]]=[[x_\lambda y]_{\lambda+\mu}z]$.
\end{itemize}

Recall that a {\bf pre-Lie algebra} is a pair $(A, \circ)$, where $A$ is a vector space and $\circ : A \otimes A \to A$ is a bilinear multiplication satisfying
\begin{equation}\label{eq:pre-Lie law}
	(x \circ y) \circ z - x \circ (y \circ z) = (y \circ x) \circ z - y \circ (x \circ z),\quad \forall x, y, z \in A.
\end{equation}
Let $(A,\circ)$ be a pre-Lie algebra. Then the new bracket $[\cdot,\cdot]:A\otimes A\to A$ defined by
\begin{equation}\label{eq:pre-Lie-Lie}
	[x,y]=x\circ y-y\circ x,\quad \forall x,y\in A
\end{equation}
makes $A$ become a Lie algebra, which is called the subadjacent Lie algebra of the pre-Lie algebra $(A,\circ)$. A \textbf{differential pre-Lie algebra} is a pre-Lie algebra $(A, \circ)$ equipped with a linear map $\partial : A \to A$ satisfying	$\partial(x \circ  y) = (\partial x) \circ y + x \circ  (\partial y)$ for all $x,y\in A$. 

As the properties of the identity element are not relevant to this paper, all definitions are considered in the non-unital case. The definition of a non-unital vertex algebra introduced by Bakalov and Kac is as follows.

\begin{defi}{\rm (\cite{Bakalov})}\label{def:VA}
	A \textbf{non-unital vertex algebra} is a quadruple $(A,  \circ, [\cdot\,_\lambda\,\cdot],\partial)$, where $(A, [\cdot\,_\lambda\,\cdot], \partial)$ is a Lie conformal algebra and $(A,\circ, \partial)$ is a differential pre-Lie algebra, such that the following identities hold:
	\begin{align}
		\label{eq:VA1}x \circ y - y \circ x &=\int_{-\partial}^{0} [x_{\lambda}y] \, d\lambda,\\
		\label{eq:VA2}[x_{\lambda}(y \circ z)] - [x_{\lambda}y] \circ z - y \circ [x_{\lambda}z] &=\int_{0}^{\lambda} [[x_{\lambda}y]_{\mu}z] \, d\mu, \quad \forall \, x, y, z \in A.
	\end{align}
\end{defi}

In the following, we recall the second definition of non-unital vertex algebra, which is also called vertex algebra without vacuum. 
\begin{defi}{\rm(\cite{FLM,HuL,Kac})}\label{def:VA2}
A {\bf non-unital vertex algebra } is a vector space $ V$, equipped with a linear map 
$$ Y : V \to ({\rm End~ V})[[t, t^{-1}]],\quad x\mapsto Y(x,t)=\sum_{n\in \mathbb{Z}} x_{(n)} t^{-n-1},\quad  x_{(n)}\in {\rm End~ V}$$ 
called the state-field correspondence, and a linear operator $ D : V \to V $ satisfying: 
\begin{itemize}
	\item[\rm(i)] {\rm (truncation property)} For any given \(x, y \in V\), $x_{(n)} y = 0$ when $n$ is sufficiently large,
	\item[\rm(ii)]{\rm (Jacobi identity)} For $x, y \in V$,
	$$
	t_0^{-1} \delta \left( \frac{t_1 - t_2}{t_0} \right) Y(x, t_1) Y(y, t_2) - t_0^{-1} \delta \left( \frac{-t_2 + t_1}{z_0} \right) Y(y, t_2) Y(x, t_1) = t_2^{-1} \delta \left( \frac{t_1 - t_0}{t_2} \right) Y(Y(x, t_0)y, t_2),
	$$
	\item[\rm(iii)]{\rm ($D$-derivative property)} For $x \in V$, $Y(Dx, t) = \frac{d}{dt} Y(x, t)$,
	\item[\rm(iv)]{\rm (skewsymmetry)}	$Y(x, t)y = e^{zD}Y(y, -t)x.\quad x,y\in V$.
\end{itemize}
We denote a non-unital vertex algebra by $ (V, Y, D) $.
	\end{defi}

The equivalence between the non-unital vertex algebra defined via the $\lambda$-bracket language (Definition \ref{def:VA}) and the non-unital vertex algebra defined via the state-field correspondence (Definition \ref{def:VA2}) is given by  

$$
[x_\lambda y] = \sum_{n \in \mathbb{Z}_+} \frac{\lambda^n}{n!} x_{(n)} y, \quad 
x \circ y = x_{(-1)} y, \quad  x_{(-n-1)} y=\frac{1}{n!}(\partial^n x) \circ y \quad (n \in \mathbb{Z}_+),\quad 
\partial = D.
$$

For further details on the equivalence among various definitions of vertex algebras, see \cite{Bakalov, DK06}.
 
Recall \cite{HL} that a {\bf pre-Lie conformal algebra} (also called a left-symmetric conformal algebra) is an $\mathbb{C}[\partial]$-module $A$
endowed with a $\lambda$-operation $\cdot\,_\lambda\,\cdot:\,A\otimes A\to A[\lambda]$
satisfying: for $x,y,z\in A$, 
\begin{align}
	(\partial x)_\lambda y = -\lambda x_\lambda y,&\quad x_\lambda (\partial y) = (\partial + \lambda)x_\lambda y,\label{eq:conf}\\
	(x_\lambda y)_{\lambda+\mu} z - x_\lambda (y_\mu z) &= (y_\mu x)_{\lambda+\mu} z - y_\mu (x_\lambda z). \label{eq:lsc}
\end{align}
Let $A$ be a pre-Lie conformal algebra. Then the new $\lambda$-bracket $[\cdot\,_\lambda\,\cdot]:A\otimes A\to A[\lambda]$ defined by 
\begin{equation}\label{eq:PLCA-LCA}
	[x_\lambda y] = x_\lambda y - y_{-\lambda-\partial}x, \quad \forall x, y \in A
\end{equation}
makes $A$ become a Lie conformal algebra, which is called a subadjacent Lie conformal algebra of $A$.

Recall \cite{BLN} that an {\bf L-dendriform algebra} is a vector space $A$ equipped with two binary operations $\triangleright,\triangleleft
: A \otimes A \to A$ satisfying 
\begin{eqnarray}
	\label{eq:L-den1}	x \triangleright (y \triangleright z) &=& (x \triangleright y) \triangleright z + (x \triangleleft y) \triangleright z +y \triangleright (x \triangleright z)- (y \triangleleft x) \triangleright z - (y \triangleright x) \triangleright z,\\
	\label{eq:L-den2}	x \triangleright (y \triangleleft z) &=& (x \triangleright y) \triangleleft z + y \triangleleft (x \triangleright z) + y \triangleleft (x \triangleleft z) - (y \triangleleft x) \triangleleft z.
\end{eqnarray}
Let $(A, \triangleright, \triangleleft)$ be an $L$-dendriform algebra. Then the linear map $\circ : A \otimes A \to A$ defined by  
\begin{equation}\label{eq:pre-Lie op}
	x \circ y = x \triangleright y + x \triangleleft y, \, \forall x, y \in A,
\end{equation}
makes $A$ become a pre-Lie algebra. 

\begin{defi}
	A {\bf differential L-dendriform algebra} is a quadruple $(A, \triangleright, \triangleleft, \partial)$, where $(A, \triangleright, \triangleleft)$ is an L-dendriform algebra and $\partial: A \to A$ is a linear operator satisfying 
	$$
	\partial (x \triangleright y) = (\partial x) \triangleright y + x \triangleright (\partial y),\qquad 
	\partial (x \triangleleft y) = (\partial x) \triangleleft y + x \triangleleft (\partial y),\quad \forall x,y\in A.
	$$
\end{defi}
It is clear that a differential L-dendriform algebra gives a differential pre-Lie algebra $(A,\circ,\partial)$, where the pre-Lie operation $\circ$ is given by \eqref{eq:pre-Lie op}.

Combining pre-Lie conformal algebras and differential L-dendriform algebras, we give the notion of a non-unital pre-vertex algebra.
\begin{defi}\label{def:pre-VA}
	A \textbf{ non-unital  pre-vertex algebra} is a quintuple $(A,\triangleright,\triangleleft,\cdot\,_{\lambda}\,\cdot,\partial)$, where $(A,\cdot\,_{\lambda}\,\cdot,\partial)$ is a pre-Lie conformal algebra and $(A,\triangleright,\triangleleft,\partial)$ is a differential L-dendriform algebra, such that the following identities hold: for $x,y,z\in A$
	\begin{eqnarray}
		\label{eq:pre-VA1}x\triangleright y-y\triangleleft x&=&\int_{-\partial}^{0} x_{\lambda}y \, d\lambda,\\
		\label{eq:pre-VA2}(y\triangleright z+y\triangleleft z)_{-\lambda - \partial} x&=&(y_{-\lambda - \partial} x)\triangleleft z+y\triangleright(z_{-\lambda - \partial} x)-\int_{0}^{\lambda}z_{-\mu - \partial}(y_{-\lambda - \partial} x) \,d\mu,\\
		\label{eq:pre-VA3}x_{\lambda} (y\triangleright z)&=&(x_{\lambda} y-y_{-\lambda - \partial} x)\triangleright z+y\triangleright (x_{\lambda} z)+\int_{0}^{\lambda}(x_{\lambda} y-y_{-\lambda - \partial} x)_\mu z\,d\mu,\\
		\label{eq:pre-VA4}x_{\lambda}(y\triangleleft z)&=&y\triangleleft (x_{\lambda} z-z_{-\lambda - \partial} x)+(x_{\lambda}y)\triangleleft z-\int_{0}^{\lambda}z_{-\mu - \partial}(x_{\lambda}y)\,d\mu.
	\end{eqnarray}
\end{defi}

\begin{pro}\label{prop:pre-vertex make vertex}
	Let $(A,\triangleright,\triangleleft,\cdot\,_{\lambda}\,\cdot,\partial)$ be a non-unital  pre-vertex algebra. Define
	$$[x_{\lambda}y]=x_{\lambda} y-y_{-\lambda - \partial} x\quad \mbox{and} \quad x \circ y=x\triangleright y+x\triangleleft y.$$
	Then $(V, \circ, [\cdot\,_{\lambda}\,\cdot],\partial)$ is a non-unital vertex algebra, which is called a subadjacent non-unital vertex algebra of $A$.
\end{pro}
\begin{proof}
	Since $(A,\cdot\,_{\lambda}\,\cdot,\partial)$ is a pre-Lie conformal algebra, $(A,[\cdot\,_{\lambda}\,\cdot],\partial)$ is a Lie conformal algebra. Since $(A,\triangleright,\triangleleft,\partial)$ is a differential L-dendriform algebra, $(A,\circ,\partial)$ is a differential pre-Lie algebra. In the following, we only need to show that the compatibility identities \eqref{eq:VA1} and \eqref{eq:VA2} hold. First, by \eqref{eq:pre-VA1}, we have
	$$
	x\circ y-y\circ x=x\triangleright y+x\triangleleft y-y\triangleright x-y\triangleleft x=(x\triangleright y-y\triangleleft x)-(y\triangleright x-x\triangleleft y)=\int_{-\partial}^{0} x_{\lambda}y \, d\lambda-\int_{-\partial}^{0} y_{\lambda}x \, d\lambda.
	$$
	Letting $\lambda=-\mu-\partial$, we have  $\int_{-\partial}^{0} y_{\lambda}x \, d\lambda=\int_{-\partial}^{0} y_{-\mu-\partial}x \, d\mu$ and thus 
	$$x\circ y-y\circ x=\int_{-\partial}^{0} x_{\lambda}y \, d\lambda-\int_{-\partial}^{0} y_{-\lambda-\partial}x \, d\lambda=\int_{-\partial}^{0} [x_{\lambda}y] \, d\lambda,$$
	which implies that \eqref{eq:VA1} holds. 
	
	By \eqref{eq:pre-VA2}-\eqref{eq:pre-VA4}, we have
	\begin{eqnarray*}
		&&[x_{\lambda}(y \circ z)] - [x_{\lambda}y] \circ z - y \circ [x_{\lambda}z] \\
		&=&\Big( x_\lambda (y \triangleright z) - (x_\lambda y) \triangleright z +(y_{-\lambda-\partial}x) \triangleright z \Big) + \Big( x_\lambda (y \triangleleft z) -  y \triangleleft (x_\lambda z) +y \triangleleft (z_{-\lambda-\partial }x) \Big) \\
		&& + \Big((y_{-\lambda-\partial}x) \triangleleft z+y \triangleright (z_{-\lambda-\partial}x) -(y \triangleright z)_{-\lambda-\partial} x -(y \triangleleft z)_{-\lambda-\partial} x\Big)-(x_\lambda y)\triangleleft z-y\triangleright (x_\lambda z)\\
		&=&\Big(y\triangleright (x_{\lambda} z)+\int_{0}^{\lambda}(x_{\lambda} y-y_{-\lambda - \partial} x)_\mu z\,d\mu\Big)+\Big((x_{\lambda}y)\triangleleft z-\int_{0}^{\lambda}z_{-\mu - \partial}(x_{\lambda}y)\,d\mu\Big)\\
		&&+\int_{0}^{\lambda}z_{-\mu - \partial}(y_{-\lambda - \partial} x) \,d\mu-y\triangleright (x_{\lambda} z)-(x_{\lambda}y)\triangleleft z\\
		&=&\int_0^\lambda (x_\lambda y-y_{-\lambda-\partial})_\mu z \, d\mu - \int_0^\lambda z_{-\mu-\partial} (x_\lambda y) \, d\mu +\int_0^{\lambda} z_{-\mu-\partial} (y_{-\lambda-\partial} x) \, d\mu\\
		&=&\int_{0}^{\lambda} [[x_{\lambda}y]_{\mu}z] \, d\mu.
	\end{eqnarray*}	
	which implies that \eqref{eq:VA2} holds.
	Thus, $(V, \circ, [\cdot\,_{\lambda}\,\cdot],\partial)$  satisfies all axioms of a vertex algebra. 
\end{proof}

Recall \cite{Loday} that a {\bf Zinbiel algebra} is a vector space $A$ together with a bilinear map  
$*: A \otimes A \to A$ satisfying
\begin{equation}\label{eq:Zinbiel alg.}
	x * (y * z) = (x * y) * z+(y * x) * z, \quad \forall x, y, z \in A.
\end{equation}

Note that a  Zinbiel algebra $(A,\ast)$ is a L-dendriform algebra with the operations $\triangleright,\triangleleft
: A \otimes A \to A$ defined by  
$$x \triangleright y = y \triangleleft x=x\ast y,\quad \forall x,y\in A.$$
Thus L-dendriform algebras can be seen as  non-commutative Zinbiel algebras. 

\begin{defi}
	A \textbf{differential Zinbiel algebra} is a Zinbiel  algebra $(A, *)$ equipped a linear map $\partial : A \to A$ satisfying
	\begin{equation}
		\partial(x * y) = (\partial x) * y + x * (\partial y), \quad \forall x, y \in A.
	\end{equation}
\end{defi} 
A differential Zinbiel algebra $(A,\ast)$ is also a differential L-dendriform algebra and thus  differential L-dendriform algebras can be seen as  non-commutative differential Zinbiel algebras.

Let $(A,\ast,\partial)$ be a differential Zinbiel algebra. Then $(A,\cdot,\partial)$ is a differential algebra, where the operation $\cdot:A\otimes A\to A$ is given by 
\begin{equation}\label{eq:Zinbiel-associative}
	x\cdot y=x* y+y*x,\quad \forall x,y\in A.
\end{equation}

\begin{ex}
	Any differential Zinbiel algebra is a non-unital pre-vertex algebra in which the pre-Lie conformal algebra structure is trivial.
\end{ex}

\begin{defi}
	Let $(A,\circ,\partial)$ be a differential pre-Lie algebra. If  a linear map $T: A \to A$ satisfies $T\circ \partial=\partial\circ T$ and 
	\begin{equation}\label{eq:RB on differential pre-alg}
		T(x) \circ T(y) = T\bigl(T(x) \circ y + x \circ T(y)\bigr), \quad \forall x, y \in A,
	\end{equation}
	then $T$ is called a {\bf  Rota-Baxter operator} on the differential pre-Lie algebra $A$.
\end{defi}

\begin{pro}\label{pre-Lie to L-dendriform}
	Let $T$ be a Rota-Baxter operator on a differential pre-Lie algebra $A$. Then the binary operations given by  
	\begin{equation}\label{eq:RB-triangle}
		x \triangleright y = T(x) \circ y, \quad x \triangleleft y = x \circ T(y), \quad \forall x, y \in A,
	\end{equation}
	define a differential  $L$-dendriform algebra structure on $A$.
\end{pro}
\begin{proof}
	It follows by a direct calculation. 
\end{proof}	

\begin{defi}{\rm (\cite{HB2})}
	Let \((A, [\cdot\,_{\lambda}\, \cdot],\partial)\) be a Lie conformal algebra.
	If $T: A \to A$ is a $\mathbb{C}[\partial]$-module homomorphism satisfying
	\begin{equation}\label{def:baxter-lie-conformal}
		[T(x)_\lambda T(y)] = T\bigl([T(x)_\lambda y] + [x_\lambda T(y)]\bigr), \quad \forall x, y \in A,
	\end{equation}
	then $T$ is called a {\bf Rota-Baxter operator} on the Lie conformal algebra $A$. 
\end{defi}

Let $T$ be a Rota-Baxter operator on the Lie conformal algebra $A$. Then the new $\lambda$-operation $\cdot\,_\lambda\,\cdot:A\otimes A\to A[\lambda]$ defied by 
\begin{equation}\label{eq:RB-pre-Lie}
	x{_\lambda}y\ = [T(x)_\lambda y], \quad \forall x, y \in A,
\end{equation}
makes $A$ become a pre-Lie conformal algebra (\cite{HB2}).

\begin{defi}
	Let $(A, \circ,[\cdot\,_{\lambda}\,\cdot],  \partial)$ be a non-unital vertex algebra and $T: A \to A$ be a $\mathbb{C}[\partial]$-module homomorphism. If $T$ is both a Rota-Baxter operator on the Lie conformal algebra $(A, [\cdot\,_{\lambda}\,\cdot],\partial)$ and the differential pre-Lie algebra $(A, \circ, \partial)$, we call it a \textbf{Rota-Baxter operator} on the  non-unital vertex algebra $A$.
\end{defi}
\begin{pro}\label{pro:RB-embed}
	Let $(A,\circ,[\cdot\,_{\lambda}\,\cdot],  \partial)$ be a non-unital vertex algebra and $T : A \to A$ a Rota-Baxter operator on A. Then $(A,\triangleright,\triangleleft,\cdot\,_{\lambda}\,\cdot,\partial)$ is a pre-vertex algebra, where the $\lambda$-operation $\cdot\,_{\lambda}\,\cdot$ is given by \eqref{eq:RB-pre-Lie} and the operations $\triangleright,\triangleleft$ are given by \eqref{eq:RB-triangle}. 
	
	Conversely, for every pre-vertex algebra $(A,\cdot\,_{\lambda}\,\cdot,\triangleright,\triangleleft,\partial)$, there exists a non-unital vertex algebra $(\mathcal{A} ,\circ_A, [\cdot\,_{\lambda}\,\cdot]_\huaA, \tilde{\partial})$ with a Rota-Baxter operator $T:\mathcal{A}\rightarrow \mathcal{A}$ such that $A$ can be identified as a subalgebra of the pre-vertex algebra obtained from $T$ on $\mathcal{A}$. 
\end{pro}
\begin{proof}
	By the fact that $T$ is a Rota-Baxter operator on the differential pre-Lie  algebra $(A,\circ,\partial)$, $(A,\triangleright,\triangleleft,\partial)$ is a differential L-dendriform algebra. By the fact that $T$ is a Rota-Baxter operator on the Lie conformal algebra $(A,[\cdot\,_{\lambda}\,\cdot],\partial)$, $(A\cdot\,_{\lambda}\,\cdot,\partial)$ is a pre-Lie conformal algebra. We only need to show that the compatibility conditions \eqref{eq:pre-VA1}-\eqref{eq:pre-VA4} hold. 
	
	By the  fact that $T$ is a Rota-Baxter operator, we have
	\begin{eqnarray*}
		x \triangleright y - y \triangleleft x = T(x) \circ y - y \circ T(x)=\int_{-\partial}^0 [T(x)_\lambda y] \, d\lambda = \int_{-\partial}^0 x_\lambda y \, d\lambda.
	\end{eqnarray*}
	which implies that \eqref{eq:pre-VA1} holds.
	
	By \eqref{eq:VA2} and the  fact that $T$ is a Rota-Baxter operator, we have
	\begin{eqnarray*}
		&&(y \triangleright z + y \triangleleft z)_{-\lambda-\partial} x-(y_{-\lambda-\partial} x) \triangleleft z- y \triangleright (z_{-\lambda-\partial} x)+\int_0^\lambda z_{-\mu-\partial} (y_{-\lambda-\partial} x) \, d\mu\\
		&=&[T\big(T(y)\circ z+y\circ T(z)\big)_{-\lambda-\partial}x]-[T(y)_{-\lambda-\partial}x]\circ T(z)-T(y)\circ[T(z)_{-\lambda-\partial}x]\\
		&&\quad+\int_0^\lambda\big[T(z)_{-\mu-\partial}[T(y)_{-\lambda-\partial}x]\big] \,d\mu\\
		&=&-[x_\lambda (T(y) \circ T(z))]+[x_\lambda T(y)] \circ T(z)+T(y) \circ [x_\lambda T(z)]+\int_0^\lambda [[x_\lambda T(y)]_\mu T(z)] \, d\mu\\
		&=&0.
	\end{eqnarray*}
	which implies that \eqref{eq:pre-VA2} holds.
	
	Similarly, we have
	\begin{eqnarray*}
		&&x_{\lambda} (y\triangleright z)-(x_{\lambda} y-y_{-\lambda - \partial} x)\triangleright z-y\triangleright (x_{\lambda} z)-\int_{0}^{\lambda}(x_{\lambda} y-y_{-\lambda - \partial} x)_\mu z\,d\mu\\
		&=&\big[T(x)_{\lambda}\big(T(y)\circ z\big)\big]-T\big([T(x)_{\lambda}y-T(y)_{-\lambda-\partial}x]\big)\circ z-T(y)\circ[T(x)_{\lambda}z]\\
		&&\quad-\int_{0}^{\lambda}\big[T\big([T(x)_{\lambda}y-T(y)_{-\lambda-\partial}x]\big)_{\mu}z\big]\,d\mu\\
		&=&\big[T(x)_{\lambda}\big(T(y)\circ z\big)\big]-T\big([T(x)_{\lambda}y+x_{\lambda}T(y)\big]\circ z-T(y)\circ[T(x)_{\lambda}z]\\
		&&\quad-\int_{0}^{\lambda}\big[T\big([T(x)_{\lambda}y+x_{\lambda}T(y)]\big)_{\mu}z\big]\,d\mu\\
		&=&\big[T(x)_{\lambda}\big(T(y)\circ z\big)\big]-[T(x)_{\lambda}T(y)]\circ z-T(y)\circ[T(x)_{\lambda}z]-\int_{0}^{\lambda}\big[[T(x)_{\lambda}T(y)]_{\mu}z\big]\,d\mu\\
		&=&0.
	\end{eqnarray*}
	which implies that \eqref{eq:pre-VA3} holds.
	
	Furthermore, we also have
	\begin{eqnarray*}
		&&x_{\lambda}(y\triangleleft z)-y\triangleleft (x_{\lambda} z-z_{-\lambda - \partial} x)-(x_{\lambda}y)\triangleleft z+\int_{0}^{\lambda}z_{-\mu - \partial}(x_{\lambda}y)\,d\mu\\
		&=&\big[T(x)_{\lambda}\big(y\circ T(z)\big)\big]-y\circ T\big([T(x)_{\lambda}z-T(z)_{-\lambda-\partial}x]\big)-[T(x)_{\lambda}y]\circ T(z)\\
		&&\quad+\int_{0}^{\lambda}\big[T(z)_{-\mu-\partial}[T(x)_{\lambda}y]\big]\,d\mu\\
		&=&\big[T(x)_{\lambda}\big(y\circ T(z)\big)\big]-y\circ T\big([T(x)_{\lambda}z+x_{\lambda}T(z)]\big)-[T(x)_{\lambda}y]\circ T(z)\\
		&&\quad -\int_{0}^{\lambda}\big[[T(x)_{\lambda}y]_{\mu}T(z)\big]\,d\mu\\
		&=&\big[T(x)_{\lambda}\big(y\circ T(z)\big)\big]-y\circ [T(x)_{\lambda}T(z)]-[T(x)_{\lambda}y]\circ T(z)-\int_{0}^{\lambda}\big[[T(x)_{\lambda}y]_{\mu}T(z)\big]\,d\mu\\
		&=&0.
	\end{eqnarray*}
	which implies that \eqref{eq:pre-VA4} holds.

Conversely,	let $A'$ be a copy of $A$, and set $\mathcal{A} := A \oplus A'$.  For $x \in A$, denote by $x'$ the corresponding element in $A'$.  Define a linear map
	$$\tilde{\partial} : \mathcal{A}\to \mathcal{A},\quad \tilde{\partial}(x + y') := \partial x + (\partial y)', \quad \forall~x, y \in A,$$
	and define a bilinear operation $\circ_A$ and a $\lambda$-bracket $[\cdot\,_\lambda \,\cdot]_\mathcal{A}$ on $A$ by  
	$$x \circ_\mathcal{A} y := x \circ y, \quad x \circ_\mathcal{A} y' := (x \triangleright y)', \quad x' \circ_\mathcal{A} y := (x \triangleleft y)', \quad x' \circ_\mathcal{A} y' := 0,$$  
	and  
	$$[x_\lambda y]_\mathcal{A} := [x_\lambda y], \quad [x_\lambda y']_\mathcal{A} := (x_\lambda y)', \quad [x'_\lambda y]_\mathcal{A} := -(y _{-\lambda - \partial} x)', \quad [x'_\lambda y']_\mathcal{A} := 0.$$
	It is straightforward to check  that $(\mathcal{A}, \circ_\mathcal{A}, [\cdot\,_\lambda\, \cdot]_\mathcal{A}, \tilde{\partial})$ is a non-unital vertex algebra.
	
	Now define a linear map $T : \mathcal{A} \to \mathcal{A}$ by  
	$$T|_{A} = 0, \quad T(x') = x, \quad \forall~x \in A.$$
	It is clear that $T \tilde{\partial} = \tilde{\partial} T$ and  $T$ is  a Rota–Baxter operator on the differential pre-Lie algebra $(\mathcal{A}, \circ_\mathcal{A}, \tilde{\partial})$. In fact, $T$ is  also a Rota–Baxter operator on the Lie conformal algebra $(\mathcal{A}, [\cdot\,_\lambda\, \cdot]_\mathcal{A}, \tilde{\partial})$. Because  $T|_{A} = 0$, the condition \eqref{def:baxter-lie-conformal} follows immediately for all cases where $x,y\in A$, or $x\in A,y'\in A'$, or $x'\in A',y\in A$.  For $x', y' \in A'$,  we have  
	$$[T(x')_{\lambda} T(y')]_{\mathcal{A}} = [x_{\lambda} y],$$
	while  
	$$[T(x')_{\lambda} y']_{\mathcal{A}} = (x_{\lambda} y)', \quad [x'_{\lambda} T(y')]_{\mathcal{A}} = -(y_{-\lambda - \partial} x)'.$$
	Hence, we have
	$$T\big([T(x')_{\lambda} y']_{\mathcal{A}} + [y'_{\lambda} T(x')]_{\mathcal{A}}\big) = T\big((x_{\lambda} y)' - (y_{-\lambda- \partial}  x)'\big) = [x_{\lambda} y] = [T(x')_{\lambda} T(y')]_{\mathcal{A}},$$
	which implies that $T$ is a Rota–Baxter operator on the Lie conformal algebra $\mathcal{A}$. Consequently,  $T$ is  a Rota–Baxter operator on the non-unital vertex algebra $(\mathcal{A}, \circ_\mathcal{A}, [\cdot\,_\lambda\, \cdot]_\mathcal{A}, \tilde{\partial})$. Moreover, the pre-vertex algebra induced by the Rota–Baxter operator $T$ on the non-unital vertex algebra $\mathcal{A}$ is given by
\begin{align*}
		x \triangleright y&=0,&x \triangleright y'&=0,&x' \triangleright y'&=(x\triangleright  y)',& x' \triangleright y&=x\circ y,\\
		x \triangleleft y&=0,&x \triangleleft y'&=x\circ y,&x' \triangleleft  y'&=(x\triangleleft  y)',& x' \triangleleft y&=0,\\
		x_\lambda y&=0,&x_\lambda y'&=0,&x'_\lambda y&=[x_\lambda y],&x'_\lambda y'&=(x_\lambda y)'.
\end{align*}
Define an embedding $j:A\rightarrow \mathcal{A}$ by $j(x)=x'$ for all $x\in A$. It is clear that 
\begin{align*}
	j(x  \triangleright y)&=(x  \triangleright y)'=x' \triangleright y'=j(x)\triangleright  j(y),\\
		j(x  \triangleleft y)&=(x  \triangleleft y)'=x' \triangleleft y'=j(x)\triangleleft  j(y),\\
			j(x _\lambda y)&=(x_\lambda y)'=x'_\lambda y'=j(x)_\lambda j(y).
	\end{align*}
Thus $A$ can be identified as a subalgebra of the pre-vertex algebra on $\mathcal{A}$ induced by the Rota–Baxter operator $T$. 
 \end{proof}

\subsection{Equivalence between pre-vertex algebras and dendriform vertex algebras}
Inspired by the axioms of dendriform algebras and weak-associativity of vertex algebras, Bai, Guo and Liu introduce the notion of a dendriform vertex algebra. 

\begin{defi}{\rm (\cite{BGL2})}
A {\bf dendriform vertex algebra} $(V,\prec_t,\succ_t,D)$ is a vector space $V$ equipped with a linear map $D:V\to V$ and two linear operators
$$
\prec_t,\ \succ_t:V\to \operatorname{Hom}\Bigl(V,V\bigl((t)\bigr)\Bigr),
$$
satisfying
\begin{eqnarray}
	(t_0+t_2)^{N_{x,z}}(x\prec_{t_0} y)\prec_{t_2}z&=&(t_0+t_2)^{N_{x,z}}x\prec_{t_0+t_2} (y\succ_{t_2}z+y\prec_{t_2} z),\\
		(t_0+t_2)^{N_{x,z}}(x\succ_{t_0} y)\prec_{t_2}z&=&(t_0+t_2)^{N_{x,z}}x\succ_{t_0+t_2} (y\prec_{t_2} z),\\
		(t_0+t_2)^{N_{x,z}}(x\succ_{t_0}y+x\prec_{t_0} y)\succ_{t_2}z&=&(t_0+t_2)^{N_{x,z}}x\succ_{t_0+t_2} (y\succ_{t_2} z),\\
	(t_1-t_2)^{N_{x,y}}x\succ_{t_1} (y\prec_{t_2}z)&=&(t_1-t_2)^{N_{x,y}}y\prec_{t_2} (x\succ_{t_1}z+x\prec_{t_1} z),\\
	(t_1-t_2)^{N_{x,y}}x\succ_{t_1} (y\succ_{t_2}z)&=&(t_1-t_2)^{N_{x,y}}y\succ_{t_2} (x\succ_{t_1}z),
	\end{eqnarray}
where $x,y,z\in V$ and $N_{x,y},N_{x,z}\in \Nat$ depend on $x,y$ and $x,z$, respectively, and moreover, they need to satisfy
\begin{eqnarray}\label{field skew-sym}
&&	e^{tD}(x\prec_{-t}y)=y\succ_t x,\qquad e^{tD}(x\succ_{-t}y)=y\prec_t x,\\
\label{D-compatibility conditions}
&&	D(x\prec_t y)-x\prec_t(Dy)=\frac{d}{dt}(x\prec_t y),\qquad
	D(x\succ_t y)-x\succ_t(Dy)=\frac{d}{dt}(x\succ_t y).
\end{eqnarray}
\end{defi}

By the skew-symmetry identities \eqref{field skew-sym}, the $ D $-bracket derivative identities \eqref{D-compatibility conditions} are equivalent to the $ D $-derivative identities
\begin{equation}\label{D-compatibility conditions2}
	(Dx) \prec_t y = \frac{d}{dt}(x \prec_t y), \quad (Dx) \succ_t y = \frac{d}{dt}(x \succ_t y).
\end{equation}

An analog of the Jacobi identity of vertex algebras, they give a equivalent description of dendriform vertex algebras.
\begin{defi}{\rm (\cite{BGL2})}\label{def:jac-dendriform}
	A {\bf dendriform vertex algebra} $(V,\prec_t,\succ_t,D)$ is a vector space $V$ equipped with a linear map $D:V\to V$ and two linear operators
	$$
	\prec_t,\ \succ_t:V\to \operatorname{Hom}\Bigl(V,V\bigl((t)\bigr)\Bigr),
	$$ 
	satisfying  \eqref{field skew-sym}, \eqref{D-compatibility conditions} and the Jacobi identity
	\begin{eqnarray}\label{Jacobi-1}
		&&t_0^{-1}\delta\!\left(\frac{t_1-t_2}{t_0}\right)x\succ_{t_1}(y\succ_{t_2}z)
		-t_0^{-1}\delta\!\left(\frac{-t_2+t_1}{t_0}\right)y\succ_{t_2}(x\succ_{t_1}z)\nonumber\\
		&=&t_2^{-1}\delta\!\left(\frac{t_1-t_0}{t_2}\right)(x\succ_{t_0}y+x\prec_{t_0}y)\succ_{t_2}z.
	\end{eqnarray}
	\end{defi}

Note that the Jacobi identity \eqref{Jacobi-1} is equivalent to one of the following identities
\begin{eqnarray}
	\label{Jacobi-2}&&t_0^{-1}\delta\!\left(\frac{t_1-t_2}{t_0}\right)x\succ_{t_1}(y\prec_{t_2}z)
	-t_0^{-1}\delta\!\left(\frac{-t_2+t_1}{t_0}\right)y\prec_{t_2}(x\succ_{t_1}z+x\prec_{t_1}z)\nonumber\\
	&=&t_2^{-1}\delta\!\left(\frac{t_1-t_0}{t_2}\right)(x\succ_{t_0}y)\prec_{t_2}z,\\
	\label{Jacobi-3}&&t_0^{-1}\delta\!\left(\frac{t_1-t_2}{t_0}\right)x\prec_{t_1}(y\prec_{t_2}z+y\succ_{t_2}z)
	-t_0^{-1}\delta\!\left(\frac{-t_2+t_1}{t_0}\right)y\succ_{t_2}(x\prec_{t_1}z)\nonumber\\
	&=&t_2^{-1}\delta\!\left(\frac{t_1-t_0}{t_2}\right)(x\prec_{t_0}y)\prec_{t_2}z.
\end{eqnarray}

Recall \cite{BGL2}  that a linear map \( T : V \to V \) on a non-unital  vertex algebra \((V, Y, D)\) is called a {\bf Rota–Baxter operator} if  $T \circ D = D \circ T $ and 
\begin{equation}\label{eq:RB-VA2}
Y\big(T(x), t\big)T(y) = T\Bigl(Y\big(T(x), t\big)y + Y(x, t)T(y)\Bigr), \quad \forall x,y\in V.
\end{equation}

\begin{pro}\label{pro:RB-dendriform}{\rm (\cite{BGL2})}
	Let $ (V, Y, D) $ be a non-unital  vertex algebra and $ T $ a Rota--Baxter operator on $ V $. Define
	$$
	x \succ_t y := Y\big(T(x), t\big)y, \quad x \prec_t y := Y(x, t)T(y).
	$$
	Then $ (V, \prec_t, \succ_t, D) $ is a dendriform vertex algebra. 
\end{pro}

\begin{pro}\label{pro:dendriform2}{\rm (\cite{BGL2})}
Let $(V,\prec_t,\succ_t,D)$ be a dendriform vertex algebra and $ (V, Y, D) $ be the corresponding non-unital  vertex algebra. Let $W=V$, and define 
$$Y_W:V\rightarrow ({\rm End}W)[[t,t^{-1}]],\quad Y_W(x,t)y:=x\prec_t y,\quad \forall~x\in V,y\in W.$$
Then $(W,Y_W)$ is a module over $ (V, Y, D) $.
\end{pro}

\begin{cor}\label{cor:vertex-embed}
For every dendriform vertex algebra $(V,\prec_t,\succ_t,D)$, there exists a non-unital  vertex algebra $(A,Y_A,\tilde{D}) $ with a Rota-Baxter operator $T:A\rightarrow A$ such that $V$ can be identified as a subalgebra  of the dendriform algebra constructed from $T$. 
\end{cor}
\begin{proof}
	let $V'$ be a copy of $V$, and set $A := V \oplus V'$.  For $v \in V$, denote by $v'$ the corresponding element in $V'$.  Define a linear map
	$$\tilde{D} : A\to A,\quad \tilde{D}(x + y') := D x + (D y)', \quad \forall~x, y \in V,$$
	and define an operation $Y : A \to ({\rm End~ A})[[t, t^{-1}]]$ by  
	$$Y_A(x,t)y=x \succ_t y+x \prec_t y,\quad Y_A(x',t)y=(x \prec_t y)',\quad Y_A(x,t)y'=(x \succ_t y)',\quad Y_A(x',t)y'=0.$$
By Proposition \ref{pro:dendriform2}, $(A, Y_A,\tilde{\partial})$ is a non-unital vertex algebra.
	
	Now define a linear map $T : A \to A$ by  
	$$T|_{V} = 0, \quad T(v') = v, \quad \forall~v \in V.$$
	It is clear that $T$ is  a Rota–Baxter operator on the non-unital vertex algebra $(A, Y_A,\tilde{\partial})$ and the dendriform vertex algebra on $A$ induced by the Rota–Baxter operator $ T$ is given by
	\begin{align*}
	x\succ_t y &= 0,&x'\succ_t y&= Y_A\big(x, t\big)y,&x\succ_t y'&=0,& x'\succ_t y'&=(x\succ_t y)',\\
	x\prec_t y &= 0,&x'\prec_t y&= 0,&x\prec_t y'&=Y_A\big(x, t\big)y,& x'\prec_t y'&=(x\prec_t y)'.
		\end{align*}
	
	Moreover, Define an embedding $j:V\rightarrow A$ by $j(x)=x'$ for all $x\in A$. Note that 
	$$j(x)\succ_t j(y)=(x\succ_t y)'=j(x\succ_t y),\quad j(x)\prec_t j(y)=(x\prec_t y)'=j(x\prec_t y).$$
	Thus $V$ can be embedded into the dendriform algebra constructed from $T$.
\end{proof}

\begin{pro}\label{pro:twoRB-equ}
$T:V\to V$ is a Rota-Baxter operator on the non-unital vertex algebra $ (V, Y, D)$ in the sense of Definition \ref{def:VA2}  if and only if $T:V\to V$ is a Rota-Baxter operator on the non-unital vertex algebra $(V,\circ,[\cdot\,_\lambda\,\cdot],\partial=D)$ in the sense of Definition \ref{def:VA}.
\end{pro}
\begin{proof}
	Let $T:V\to V$ be a Rota-Baxter operator on the dendriform vertex algebra $ (V, Y, D)$.  Since $T\circ D=D\circ T$, we have $T\circ \partial=\partial\circ T$. Note that the condition \eqref{eq:RB-VA2} is equivalent to 
	\begin{equation}\label{eq:RBcomp}
		(T(x))_{(n)}(T(y))=T\big(T(x)_{(n)}y+x_{(n)} (T(y))\big),\quad n\in \mathbb{Z}.
	\end{equation}
	Note the Lie conformal algebra $\lambda$-bracket on $V$ is given by 
	$$[x_\lambda y]=\sum_{n\in \mathbb{Z}_+}\frac{\lambda^n}{n!}x_{(n)}y.$$
	For $n\geq 0$, \eqref{eq:RBcomp} is equivalent to
$$	[T(x)_\lambda T(y)] = T\bigl([T(x)_\lambda y] + [x_\lambda T(y)]\bigr).$$
For $n=-1$, by the definition of the pre-Lie operation  $x\circ y=x_{(-1)}y$, \eqref{eq:RBcomp} is equivalent to
$$T(x) \circ T(y) = T\bigl(T(x) \circ y + x \circ T(y)\bigr).$$
Thus $T:V\to V$ is a Rota-Baxter operator on the non-unital vertex algebra $(V,\circ,[\cdot\,_\lambda\,\cdot],\partial=\partial)$.

Conversely, note that $x_{(-n-1)}y=\frac{1}{n!}(\partial^n x )\circ y$ for $n\in \mathbb{Z}_+$. Since $T\circ \partial=\partial\circ T$ and $T$ is a Rota-Baxter operator on the pre-Lie algebra $(V,\circ)$, \eqref{eq:RBcomp} holds for any $n\leq -1$.  \eqref{eq:RBcomp} follows  for any $n\geq 0$ by the fact that $T$ is a Rota-Baxter operator on the Lie conformal algebra $(V,[\cdot\,_\lambda\,\cdot],\partial)$. Thus $T:V\to V$ is a Rota-Baxter operator on the non-unital vertex algebra $ (V, Y, D)$.
\end{proof}

Given that the dendriform vertex algebra and the non-unital pre-vertex algebra are both derived from a non-unital vertex algebra through the Rota-Baxter operator, we expect them to be equivalent. In what follows, we prove that the given dendriform vertex algebra is equivalent to a non-unital pre-vertex algebra.

Let $(V,\prec_t,\succ_t,D)$ be a dendriform vertex algebra. For $x,y\in V$, we write
$$
x\prec_t y=\sum_{n\in\mathbb Z}(x\prec_{(n)}y)t^{-n-1},\qquad
x\succ_t y=\sum_{n\in\mathbb Z}(x\succ_{(n)}y)t^{-n-1},
$$
where $\succ_{(n)},\prec_{(n)}:V\otimes V\rightarrow V$ are linear maps for all $n\in \mathbb{Z}$. 

We define
\begin{equation}\label{eq:lambda-1}
x\triangleleft y:=x\prec_{(-1)}y,\qquad
x\triangleright y:=x\succ_{(-1)}y,\qquad
x_\lambda y:=\operatorname{Res}_t e^{\lambda t}(x\succ_t y)=\sum_{n\in \mathbb{Z}_+}\frac{\lambda^n}{n!}x\succ_{(n)}y,
\end{equation}
and set $\partial:=D$. 

By Condition $e^{tD}(x\succ_{-t}y)=y\prec_t x$ in  \eqref{field skew-sym} for $x,y\in V$, we have
\begin{equation}\label{eq:lambda-2}
\operatorname{Res}_t e^{\lambda t}(x\prec_t y)=-\,y_{-\lambda-\partial}x.
\end{equation}

By \eqref{D-compatibility conditions2}, for $n\geq 0$, we have $$x\prec_{(-n-1)}y=\frac{1}{n!}(\partial^n x)\prec_{(-1)}y=\frac{1}{n!}(\partial^n x )\triangleleft y,$$ and $$x\succ_{(-n-1)}y=\frac{1}{n!}(\partial^n x)\succ_{(-1)}y=\frac{1}{n!}(\partial^n x) \triangleright y.$$

\emptycomment{Therefore,
$$
\operatorname{Res}_z e^{\lambda z}\bigl(x\succ_z y+x\prec_z y\bigr)
=
x_\lambda y-y_{-\lambda-\partial}x.
$$
We also define the total field
$$
Y(x,z)y:=x\prec_z y+x\succ_z y=\sum_{n\in\mathbb Z}(x_{(n)}y)z^{-n-1}.
$$
Hence
$$
x_{(n)}y=x\prec_{(n)}y+x\succ_{(n)}y,\qquad n\in\mathbb Z.
$$
In other words, the singular part of the total field $Y(x,z)y:=x\succ_z y+x\prec_z y$ gives precisely the subadjacent Lie conformal bracket associated to the pre-Lie conformal product $x_\lambda y$.
}

The following equivalence can be proved using the fact that both a dendriform vertex algebra and a non-unital pre-vertex algebra can be embedded into the two equivalent definitions of a non-unital vertex algebra with a Rota-Baxter operator. To make the correspondence clearer, we provide a direct proof on one side.
\begin{thm}
	 $(V,\prec_t,\succ_t,D)$ is a dendriform vertex algebra if and only if  $(V,\triangleright,\triangleleft,\cdot\,_\lambda\,\cdot,\partial=D)$ is a non-unital pre-vertex algebra.
\end{thm}
\begin{proof}
Let $(V,\prec_t,\succ_t,D)$ be a dendriform vertex algebra.	Comparing the constant terms of both sides of $e^{tD}(x\prec_{-t}y)=y\succ_t x$ and $e^{tD}(x\succ_{-t}y)=y\prec_t x$ in  \eqref{field skew-sym}, we have 
\begin{equation}\label{sym-mode-constant-1}
x\triangleleft y+\sum_{n\ge 0}\frac{(-1)^{n+1}}{(n+1)!}\partial^{n+1}(x\prec_{(n)}y)-	y\triangleright x=0 
\end{equation}
and 
\begin{equation}\label{sym-mode-constant-2}
	x\triangleright y+\sum_{n\ge 0}\frac{(-1)^{n+1}}{(n+1)!}\partial^{n+1}(x\succ_{(n)}y)-	y\triangleleft x=0 ,
\end{equation}

By  \eqref{Jacobi-1}, we  have
\begin{equation}\label{eq:Jacobi-1-mode}
x\succ_{(m)}(y\succ_{(n)}z)-y\succ_{(n)}(x\succ_{(m)}z)=\sum_{j\ge0}\binom{m}{j}(x\succ_{(j)}y+x\prec_{(j)}y)\succ_{(m+n-j)}z,
\end{equation}
Letting $m=n=-1$, by \eqref{sym-mode-constant-1} and  \eqref{sym-mode-constant-2}, we get
\begin{align*}
x\triangleright(y\triangleright z)-y\triangleright(x\triangleright z)&=\sum_{j\ge0}\binom{-1}{j}(x\succ_{(j)}y+x\prec_{(j)}y)\succ_{(-2-j)}z\\
&=\sum_{j\ge0}(-1)^{j}(x\succ_{(j)}y+x\prec_{(j)}y)\succ_{-(j+1)-1}z\\
&=\big(\sum_{j\ge0}\frac{(-1)^{j}}{(j+1)!}\partial^{(j+1)}(x\succ_{(j)}y)+\sum_{j\ge0}\frac{(-1)^{j}}{(j+1)!}\partial^{(j+1)}(x\prec_{(j)}y)\big)\succ_{(-1)} z\\
&=(x\triangleright y)\triangleright z-(y\triangleleft x)\triangleright z+(x\triangleleft y)\triangleright z-(y\triangleright x)\triangleright z,
\end{align*}
which implies that \eqref{eq:L-den1} holds.

By  \eqref{Jacobi-2}, we have
\begin{equation}\label{eq:Jacobi-2-mode}
x\succ_{(m)}(y\prec_{(n)}z)-y\prec_{(n)}(x\succ_{(m)}z+x\prec_{(m)}z)=\sum_{j\ge0}\binom{m}{j}(x\succ_{(j)}y)\prec_{(m+n-j)}z,
\end{equation}

Letting $m=n=-1$, by  \eqref{sym-mode-constant-2}, we get
\begin{align*}
	x\triangleright(y\triangleleft z)-y\triangleleft(x\triangleright z+x\triangleleft z)&=\sum_{j\ge0}\binom{-1}{j}(x\succ_{(j)}y)\prec_{(-2-j)}z=\sum_{j\ge0}(-1)^{j}(x\succ_{(j)}y)\prec_{-(j+1)-1}z\\
	&=\big(\sum_{j\ge0}\frac{(-1)^{j}}{(j+1)!}\partial^{(j+1)}x\succ_{(j)}y\big)\prec_{(-1)} z=(x\triangleright y-y\triangleleft x)\triangleleft z,
\end{align*}
which implies that \eqref{eq:L-den2} holds.

Furthermore, comparing the constant terms of both sides of $D(x\prec_t y)-x\prec_t(Dy)=(Dx) \prec_t y$ and $D(x\succ_t y)-x\succ_t(Dy)=(Dx) \succ_t y $ in \eqref{D-compatibility conditions}, we have 
		$$
		\partial (x \triangleleft y) = (\partial x) \triangleleft y + x \triangleleft (\partial y)
	,\qquad 
	\partial (x \triangleright y) = (\partial x) \triangleright y + x \triangleright (\partial y).
	$$
Thus  $(V,\triangleleft,\triangleright,\partial)$ is a differential L-dendriform algebra.

By $(Dx) \succ_t y = \frac{d}{dt}(x \succ_t y)$ in \eqref{D-compatibility conditions2}, we have
\[
(\partial x)_\lambda y
=\operatorname{Res}_t e^{\lambda t}\bigl((\partial x)\succ_t y\bigr)
=\operatorname{Res}_t e^{\lambda t}\frac{d}{dt}(x\succ_t y)
=-\lambda\operatorname{Res}_t e^{\lambda t}\bigl(x\succ_t y\bigr)=-\lambda x_\lambda y.
\]
Moreover, by $D(x\succ_t y)-x\succ_t(Dy)=\frac{d}{dz}(x\succ_t y)$ in \eqref{D-compatibility conditions}, we have
$$x_\lambda(\partial y)
=\operatorname{Res}_t e^{\lambda z}\bigl(x\succ_t(\partial y)\bigr)
=\partial(x_\lambda y)+\lambda x_\lambda y
=(\partial+\lambda)x_\lambda y.
$$
Thus the $\lambda$-operation $\cdot\,_\lambda\,\cdot$ satisfies \eqref{eq:conf}.

By  \eqref{eq:Jacobi-1-mode}, multiplying both sides by $\lambda^m/m! \ ,\mu^n/n!$ and summing over $m,n\ge0$, the left-hand side becomes
\[
x_\lambda(y_\mu z)-y_\mu(x_\lambda z).
\]
For the right-hand side, using
\[
\binom{m}{j}\frac{\lambda^m}{m!}
=
\frac{\lambda^j}{j!}\frac{\lambda^{m-j}}{(m-j)!},
\]
we obtain
$$
\sum_{j,k\ge0}\frac{\lambda^j}{j!}\frac{(\lambda+\mu)^k}{k!}
\bigl(x\succ_{(j)}y+x\prec_{(j)}y\bigr)\succ_{(k)}z=(x_\lambda y-y_{-\lambda-\partial}x)_{\lambda+\mu}z.
$$
Hence,  the $\lambda$-operation $\cdot\,_\lambda\,\cdot$ satisfies \eqref{eq:lsc}, making  $(V,\cdot\,_{\lambda}\,\cdot,\partial)$  a pre-Lie conformal algebra.

In the following, we  left to check that \eqref{eq:pre-VA1}-\eqref{eq:pre-VA4} hold.
By  $x_\lambda y=\sum_{n\geq0}\frac{\lambda^n}{n!}x\succ_{(n)}y$, we have
	$$
	\int_{-\partial}^{0}x_\lambda y\,d\lambda
	=
	\sum_{n\ge 0}\frac{(-1)^n}{(n+1)!}\partial^{n+1}(x\succ_{(n)}y).
	$$
and by 
\eqref{sym-mode-constant-2} $$x\triangleright y+\sum_{n\ge 0}\frac{(-1)^{n+1}}{(n+1)!}\partial^{n+1}(x\succ_{(n)}y)-	y\triangleleft x=0,$$
 we have 
	$$
	x\triangleright y-y\triangleleft x=\int_{-\partial}^{0}x_\lambda y\,d\lambda, 
	$$
	which implies that the compatibility Condition \eqref{eq:pre-VA1} holds.

	By  \eqref{Jacobi-3}, we have
	\begin{equation}\label{eq:Jacobi-3-mode}
	x\prec_{(m)}(y\prec_{(n)}z+y\succ_{(n)}z)-y\succ_{(n)}(x\prec_{(m)}z)
=
\sum_{j\ge0}\binom{m}{j}(x\prec_{(j)}y)\prec_{(m+n-j)}z,
	\end{equation}
	Letting $n=-1$, we get
\begin{equation}\label{eq:dend-pre vertex}
	x\prec_{(m)}(y\triangleleft z+y\triangleright z)-y\triangleright(x\prec_{(m)}z)
	=
	\sum_{j\ge0}\binom{m}{j}(x\prec_{(j)}y)\prec_{(m-1-j)}z.
\end{equation}

Multiplying both sides by $\lambda^m/m!$ and summing over $m\ge0$, by \eqref{eq:lambda-2}, the  left-hand side becomes
$$
-(y\triangleleft z+y\triangleright z)_{-\lambda-\partial}x
+
y\triangleright(z_{-\lambda-\partial}x),
$$
while the right-hand side becomes 
$$
\sum_{n\ge0}\frac{\lambda^n}{n!}
\bigl(-\,y_{-\lambda-\partial}x\bigr)\prec_{(n-1)}z=-\,(y_{-\lambda-\partial}x)\triangleleft z
+
\sum_{n\ge0}\frac{\lambda^{n+1}}{(n+1)!}
\bigl((-\,y_{-\lambda-\partial}x)\prec_{(n)}z\bigr).
$$
On the other hand, by 
	$
	\sum_{n\ge0}\frac{\mu^n}{n!}\bigl((-\,y_{-\lambda-\partial}x)\prec_{(n)}z\bigr)=
	\,z_{-\mu-\partial}(\,y_{-\lambda-\partial}x),
	$
we have
	$$
	\sum_{n\ge0}\frac{\lambda^{n+1}}{(n+1)!}
	\bigl((-\,y_{-\lambda-\partial}x)\prec_{(n)}z\bigr)
	=
	\int_0^\lambda z_{-\mu-\partial}(y_{-\lambda-\partial}x)\,d\mu.
	$$
Furthermore, comparing both sides, we obtain
	$$
	(y\triangleright z+y\triangleleft z)_{-\lambda-\partial}x
	=
	(y_{-\lambda-\partial}x)\triangleleft z
	+
	y\triangleright(z_{-\lambda-\partial}x)
	-
	\int_0^\lambda z_{-\mu-\partial}(y_{-\lambda-\partial}x)\,d\mu,
	$$
	which implies that the compatibility Condition \eqref{eq:pre-VA2} holds.
	
Setting $n=-1$ in \eqref{eq:Jacobi-1-mode}, we get
$$x\succ_{(m)}(y\triangleright z)-y\triangleright(x\succ_{(m)}z)=\sum_{j\ge0}\binom{m}{j}(x\succ_{(j)}y+x\prec_{(j)}y)\succ_{(m-1-j)}z,$$
Multiplying both sides by $\lambda^m/m!$ and summing over $m\ge0$, the left-hand side becomes
$$x_\lambda(y\triangleright z)-y\triangleright(x_\lambda z),$$
while the right-hand side becomes 
$$\sum_{n\ge0}\frac{\lambda^n}{n!}(x_\lambda y-y_{-\lambda-\partial}x)\succ_{(n-1)}z=(x_\lambda y-y_{-\lambda-\partial}x)\triangleright z+\sum_{n\ge0}\frac{\lambda^{n+1}}{(n+1)!}(x_\lambda y-y_{-\lambda-\partial}x)\succ_{(n)}z,$$
On the other hand, by $\sum_{n\ge0}\frac{\mu^n}{n!}(x_\lambda y-y_{-\lambda-\partial}x)\succ_{(n)}z=((x_\lambda y-y_{-\lambda-\partial}x)_\mu z)$, we have $$\sum_{n\ge0}\frac{\lambda^{n+1}}{(n+1)!}(x_\lambda y-y_{-\lambda-\partial}x)\succ_{(n)}z=\int_0^\lambda (x_\lambda y-y_{-\lambda-\partial}x)_\mu z\,d\mu,$$
Comparing both sides, we obtain
$$x_\lambda(y\triangleright z)=(x_\lambda y-y_{-\lambda-\partial}x)\triangleright z+y\triangleright(x_\lambda z)+\int_0^\lambda (x_\lambda y-y_{-\lambda-\partial}x)_\mu z\,d\mu,$$
which implies that the compatibility Condition \eqref{eq:pre-VA3} holds.

Letting $n=-1$ in \eqref{eq:Jacobi-2-mode}, we get
$$x\succ_{(m)}(y\triangleleft z)-y\triangleleft(x\succ_{(m)}z+x\prec_{(m)}z)=\sum_{j\ge0}\binom{m}{j}(x\succ_{(j)}y)\prec_{(m-1-j)}z.$$
Multiplying both sides by $\lambda^m/m!$ and summing over $m\ge0$, the left-hand side becomes
$$x_\lambda(y\triangleleft z)-y\triangleleft(x_\lambda z-z_{-\lambda-\partial}x),$$
while the right-hand side becomes
$$\sum_{n\ge0}\frac{\lambda^n}{n!}(x_\lambda y)\prec_{(r-1)}z=(x_\lambda y)\triangleleft z+\sum_{n\ge0}\frac{\lambda^{n+1}}{(n+1)!}(x_\lambda y)\prec_{(n)}z,$$
On the other hand, by $\sum_{n\ge0}\frac{\mu^n}{n!}(x_\lambda y)\prec_{(n)}z=\operatorname{Res}_t e^{\mu t}((x_\lambda y)\prec_t z)=-\,z_{-\mu-\partial}(x_\lambda y),$
 we have
$$\sum_{n\ge0}\frac{\lambda^{n+1}}{(n+1)!}(x_\lambda y)\prec_{(n)}z=-\int_0^\lambda z_{-\mu-\partial}(x_\lambda y)\,d\mu.$$
Comparing both sides, we obtain
$$x_\lambda(y\triangleleft z)=y\triangleleft(x_\lambda z-z_{-\lambda-\partial}x)+(x_\lambda y)\triangleleft z-\int_0^\lambda z_{-\mu-\partial}(x_\lambda y)\,d\mu, $$
which implies that the compatibility Condition \eqref{eq:pre-VA4} holds. Thus $(V,\triangleright,\triangleleft,\cdot\,_\lambda\,\cdot,\partial=D)$ is a non-unital pre-vertex algebra.

Conversely, let $(V, \triangleright, \triangleleft, \cdot\,_\lambda\,\cdot, \partial)$ be a non-unital pre-vertex algebra. By Proposition \ref{pro:RB-embed}, $V$ can be embedded into a non-unital vertex algebra $(\huaA = V \oplus V',\circ_\huaA, [\cdot\,_\lambda \,\cdot]_\huaA,  \tilde{\partial})$ equipped with a Rota–Baxter operator $T$. Moreover, by Proposition \ref{pro:twoRB-equ}, $(\huaA, Y_\huaA, \tilde{D} = \tilde{\partial})$ is a non-unital vertex algebra in the sense of Definition \ref{def:VA2}, retaining the same Rota–Baxter operator $T$. Then, by Proposition \ref{pro:RB-dendriform}, the operator $T$ induces a dendriform vertex algebra $(\huaA, \prec_t, \succ_t, \tilde{D})$ on $\huaA$. Restricting this structure to $V' \simeq V$ yields the desired dendriform vertex algebra on $V$. 
\end{proof}

\section{Pre-Poisson vertex algebras and Poisson vertex algebras}\label{sec:pre-Poisson}
In this section, we first recall the notion of a Poisson vertex algebra. We then introduce pre-Poisson vertex algebras and show that every pre-Poisson vertex algebra naturally gives rise to a Poisson vertex algebra. 

Recall that a {\bf differential algebra} is a commutative associative algebra $A$ equipped with a linear map $\partial:A\to A$ satisfying $\partial(x\cdot y)=(\partial x) \cdot y+x\cdot (\partial y)$ for $x$ and $y$ in $A$.

\begin{defi}
A {\bf Poisson vertex algebra} is a differential algebra $ (P,\cdot,\partial)$
endowed with an Lie conformal algebra $\lambda$-bracket $[\cdot\,_\lambda\,\cdot]:\,P\otimes P\to P[\lambda]$
satisfying the Leibniz rule ($x,y,z\in P$):
\begin{equation}\label{eq:leib}
[x_\lambda (y\cdot z)]=y\cdot[x_\lambda z]+z\cdot [x_\lambda y]\,.
\end{equation}
\end{defi}

Analogous to the notion of pre-Poisson algebras, we will give the notion of pre-Poisson vertex algebras. 
\begin{defi}\label{defi:pre-PVA}
A {\bf pre-Poisson vertex algebra} is a quadruple $(A,\ast,\cdot\,_{\lambda}\,\cdot,\partial)$, where $(A,\ast,\partial)$ is a differential Zinbiel algebra and $(A,\cdot_{\lambda}\cdot,\partial)$ is a
	pre-Lie conformal algebra satisfying: for $x,y,z\in A$
\begin{eqnarray}
\label{eq:pre-PVA1}(x_\lambda y - y_{-\lambda-\partial} x) * z& =& x_\lambda (y * z) - y * (x_\lambda z),\\
\label{eq:pre-PVA2}(x * y + y * x)_{-\lambda-\partial} z &= &y * (x_{-\lambda-\partial} z) + x * (y_{-\lambda-\partial} z).
\end{eqnarray}
\end{defi}

Any pre-Poisson vertex algebra can give rise to a Poisson vertex algebra.
\begin{pro}
 Let $(A,\ast,\cdot\,_{\lambda}\,\cdot,\partial)$ be a pre-Poisson vertex algebra. Then $(A,\cdot,[\cdot\,_\lambda \,\cdot],\partial)$ is a Poisson vertex algebra, where the operation $\cdot:A\otimes A\to A$ is given by \eqref{eq:Zinbiel-associative} and the $\lambda$-bracket $[\cdot\,_\lambda\,\cdot]:A\otimes A\to A[\lambda]$ is given by \eqref{eq:PLCA-LCA}.
\end{pro}
\begin{proof}
Let $(A,\ast,\cdot\,_{\lambda}\,\cdot,\partial)$ be a pre-Poisson vertex algebra. Since $(A,\ast,\partial)$ is a differential Zinbiel algebra, $(A,\cdot,\partial)$ is a differential algebra. Since $(A,\cdot\,_{\lambda}\,\cdot,\partial)$ is a pre-Lie conformal algebra, $(A,[\cdot\,_\lambda \,\cdot],\partial)$ is a Lie conformal algebra. In the following, we only need to show that the Leibniz rule holds. In fact, by \eqref{eq:pre-PVA1} and \eqref{eq:pre-PVA2}, we have 
\begin{align*}
	[x_\lambda (y\cdot z)] &= x_\lambda(y*z + z*y) - (y*z+z*y)_{-\lambda-\partial}x \\
	&= \big[ (x_\lambda y - y_{-\lambda-\partial}x)*z + y*(x_\lambda z) \big] + \big[ (x_\lambda z - z_{-\lambda-\partial}x)*y + z*(x_\lambda y) \big] \\
	&\quad - \big[ z*(y_{-\lambda-\partial}x) + y*(z_{-\lambda-\partial}x) \big]  \\
	&= (x_\lambda y - y_{-\lambda-\partial}x)*z + z*(x_\lambda y) + (x_\lambda z - z_{-\lambda-\partial}x)*y + y*(x_\lambda z) \\
	&\quad - z*(y_{-\lambda-\partial}x) - y*(z_{-\lambda-\partial}x).
\end{align*}
On the other hand, we have
\begin{align*}
	[x_\lambda y]\cdot z &= (x_\lambda y - y_{-\lambda-\partial}x)\cdot z = (x_\lambda y - y_{-\lambda-\partial}x)*z + z*(x_\lambda y - y_{-\lambda-\partial}x), \\
	y\cdot [x_\lambda z] &= y\cdot (x_\lambda z - z_{-\lambda-\partial}x) = y*(x_\lambda z - z_{-\lambda-\partial}x) + (x_\lambda z - z_{-\lambda-\partial}x)*y.
\end{align*}
These equalities imply that the Leibniz rule holds.
\end{proof}
Recall \cite{Cha} that a {\bf perm algebra} is a vector space $A$ equipped with a linear operation $\circ:A\otimes A\to A$ satisfying
\begin{equation}\label{eq:perm}
	x\circ (y\circ z)=(x\circ y)\circ z=(y\circ x)\circ z,\quad \forall x,y,z\in A.
\end{equation}

Combining the perm algebra and pre-Poisson vertex algebra, we can get Poisson vertex algebras. 
\begin{pro}\label{pro:constructPVA}
	Let $(A,\ast,\cdot\,_{\lambda}\,\cdot,\partial)$ be a pre-Poisson vertex algebra and $(B,\circ)$ be a perm algebra. Define the linear operations $\partial:B\otimes A\to B\otimes A $,  $\cdot:(B\otimes A)\otimes (B\otimes A)\to B\otimes A$ and the $\lambda$-bracket $[\cdot\,_\lambda\,\cdot]:(B\otimes A)\otimes (B\otimes A)\to B\otimes A[\lambda]$ by
	\begin{eqnarray*}
		\partial(p\otimes x)&=&p\otimes (\partial x),\\	
		(p\otimes x)\cdot (q\otimes y)&=&p\circ q\otimes (x\ast y)+q\circ p\otimes (y\ast x),\\
		{	[p\otimes x_\lambda q\otimes y]}&=&p\circ q\otimes (x_\lambda y)+q\circ p\otimes (y_{-\lambda-\partial}x),\quad \forall x,y\in A, p,q\in B.
	\end{eqnarray*}
	Then $(B\otimes A,\cdot,[\cdot\,_\lambda\,\cdot],\partial)$ is a Poisson vertex algebra.
\end{pro}
\begin{proof}
	First, we show that $(B\otimes A,\cdot,\partial)$ is a differential algebra. For $p,q,t\in B$ and $x,y,z\in A$, we have
	\begin{eqnarray*}
		\big((p\otimes x)\cdot (q\otimes y)	\big)\cdot (t\otimes z)&=&(p\circ q)\circ t\otimes (x\ast y)\ast z+t\circ(p\circ q)\otimes z\ast(x\ast y)\\
		&&+(q\circ p)\circ t\otimes (y\ast x)\ast z+t\circ (q\circ p)\otimes z\ast (y\ast x);\\
		(p\otimes x)\cdot \big((q\otimes y)	\cdot (t\otimes z)\big)&=&p\circ (q\circ t)\otimes x\ast (y\ast z)+(q\circ t)\circ p\otimes (y\ast z)\ast x\\
		&&+p\circ (t\circ q)\otimes x\ast (z\ast y)+(t\circ q)\circ p\otimes (z\ast y) \ast x.
	\end{eqnarray*}
	By \eqref{eq:perm} and \eqref{eq:Zinbiel alg.},  we have 
	\begin{align*}
		t\circ(p\circ q)\otimes z\ast(x\ast y)=p\circ (t\circ q)\otimes x\ast (z\ast y)&,\\
		(p\circ q)\circ t\otimes (x\ast y)\ast z+(q\circ p)\circ t\otimes (y\ast x)\ast z-p\circ (q\circ t)\otimes x\ast (y\ast z)&=0,\\
		t\circ (q\circ p)\otimes z\ast (y\ast x)-(q\circ t)\circ p\otimes (y\ast z)\ast x-(t\circ q)\circ p\otimes (z\ast y) \ast x&=0,
	\end{align*}
	which imply that 
	$$\big((p\otimes x)\cdot (q\otimes y)	\big)\cdot (t\otimes z)=(p\otimes x)\cdot \big((q\otimes y)	\cdot (t\otimes z)\big).$$
	The commutativity of  $\cdot$ is clear. The condition $\partial(x\ast y)=\partial (x)\ast y+x\ast \partial (y)$ for $x,y\in A$ implies that
	$$\partial\big(	(p\otimes x)\cdot (q\otimes y)\big)=\partial\big(p\otimes x\big)\cdot (q\otimes y)+(p\otimes x)\cdot \partial\big(q\otimes y\big).$$
	Thus $(B\otimes A,\cdot,\partial)$ is a differential algebra. 
	
	The sesquilinearity of $\lambda$-bracket $[\cdot\,_\lambda\,\cdot]$ on $B\otimes A$ follows by the sesquilinearity of pre-Lie $\lambda$-operation $\cdot\,_\lambda\,\cdot$ on $A$. The skewsymmetry of  $\lambda$-bracket $[\cdot\,_\lambda\,\cdot]$ on $B\otimes A$ follows by the definition of this $\lambda$-bracket. In the following, we show that the $\lambda$-bracket $[\cdot\,_\lambda\,\cdot]$ on $B\otimes A$ satisfies the Jacobi identity. By \eqref{eq:perm} and the sesquilinearity of pre-Lie $\lambda$-operation $\cdot\,_\lambda\,\cdot$ on $A$, we have
	\begin{align*}
		&[p\otimes x_\lambda [q\otimes y_\mu t\otimes z]]-[[p\otimes x_\lambda q\otimes y]_{\lambda+\mu}t\otimes z]-[q\otimes y_\mu [p\otimes x_\lambda t\otimes z]]\\
		&=p\circ(q\circ t)\otimes \Big(x_\lambda(y_\mu z)-y_\mu(x_\lambda z)-(x_\lambda y)_{\lambda+\mu}z+(y_{\mu}x)_{\lambda+\mu} z\Big)\\
		&+q\circ(t\circ p)\otimes \Big(y_\mu(z_{-\lambda-\mu-\partial} x)-z_{-\lambda-\mu-\partial}(y_\mu x)-(y_\mu z)_{-\lambda-\partial}x+(z_{-\lambda-\mu-\partial}y)_{-\lambda-\partial} x\Big)\\
		&-p\circ(t\circ q)\otimes \Big(x_\lambda(z_{-\lambda-\mu-\partial} x)-z_{-\lambda-\mu-\partial}(x_\lambda y)-(x_\lambda z)_{-\mu-\partial}y+(z_{-\lambda-\mu-\partial}x)_{-\mu-\partial} y\Big)
	\end{align*}
	in which the $\partial$ acts from the left. Furthermore, by \eqref{eq:lsc}, the Jacobi identity follows. Thus $(B\otimes A,[\cdot\,_\lambda\,\cdot],\partial)$ is a Lie conformal algebra.
	
	Finally, by \eqref{eq:pre-PVA1} and \eqref{eq:pre-PVA2},  we have
	\begin{align*}
		&[p\otimes x_\lambda \big((q\otimes y)\cdot (t\otimes z)	\big)]-[p\otimes x_\lambda q\otimes y]\cdot (t\otimes z)-[p\otimes x_\lambda t\otimes z]\cdot (q\otimes y)\\
		&=p\circ(q\circ t)\otimes \Big(  x_\lambda (y * z) - y * (x_\lambda z)  -(x_\lambda y)\ast z + (y_{-\lambda-\partial} x) * z\Big)\\
		&+p\circ(t\circ q)\otimes \Big( x_\lambda (z* y) - z* (x_\lambda y)  -(x_\lambda z)\ast y + (z_{-\lambda-\partial} x) * y   \Big)\\
		&-q\circ(t\circ p)\otimes \Big( (y\ast z)_{-\lambda-\partial}x+ (z\ast y)_{-\lambda-\partial}x -z\ast(y_{-\lambda-\partial}x)-y\ast(z_{-\lambda-\partial}x) \Big)=0.
	\end{align*}
	
	Therefore, $(B\otimes A,\cdot,[\cdot\,_\lambda\,\cdot],\partial)$ is a Poisson vertex algebra.
\end{proof}

\begin{defi}
Let $(A,\cdot,\partial)$ be a differential algebra. If  a linear map $T: A \to A$ satisfies $T\circ \partial=\partial\circ T$ and 
\begin{equation}\label{eq:BO on differential alg}
T(x) \cdot T(y) = T\bigl(T(x) \cdot y + x \cdot T(y)\bigr), \quad \forall x, y \in A,
\end{equation}
then $T$ is called a {\bf  Rota-Baxter operator} on the differential algebra $A$.
\end{defi}
Let $T$ be a Rota-Baxter operator on a differential algebra $A$. It is straightforward to check that the new operation $\ast:A\otimes A\to A$ defined by
\begin{equation}\label{eq:RB-zin}
	x * y = T(x) \cdot y,\quad \forall x,y\in A
	\end{equation}
makes $A$ become a differential Zinbiel algebra with the same differential $\partial$.

\begin{defi}
	Let $(P,\cdot,[\cdot\,_\lambda\, \cdot],\partial)$ be a Poisson vertex algebra and $T:P\to P$ be a linear map. If $T$ is both a Rota-Baxter operator on the differential  algebra $(P,\cdot,\partial)$ and the Lie conformal algebra $(P,[\cdot\,_\lambda\, \cdot],\partial)$, then $T$ is called a {\bf Rota-Baxter operator} on the Poisson vertex algebra $P$.
	\end{defi}

The following proposition shows that any Rota-Baxter operator on Poisson vertex algebra can
induce a pre-Poisson vertex algebra, and any pre-Poisson vertex algebra can be embedded into a Poisson vertex algebra with a Rota-Baxter operator.

\begin{pro}
	Let $T$ be a Rota-Baxter operator on the Poisson vertex algebra $(P,\cdot,[\cdot\,_\lambda\, \cdot],\partial)$. Then $(P,\ast,\cdot\,_{\lambda}\,\cdot,\partial)$ is a pre-Poisson Vertex algebra, where the operation $\ast$ is given by \eqref{eq:RB-zin} and the $\lambda$-operation $\cdot\,_{\lambda}\,\cdot$ is given by \eqref{eq:RB-pre-Lie}.
	
		Conversely, for every pre-Poisson vertex algebra $(A,\ast,\cdot\,_{\lambda}\,\cdot,\partial)$, there exists a Poisson vertex algebra $(P,\cdot,[\cdot\,_\lambda\, \cdot],{\partial})$ with a Rota-Baxter operator $T:P\rightarrow P$ such that $A$ can be identified as a subalgebra of the pre-Poisson vertex algebra obtained from $T$. 
\end{pro}
\begin{proof}
By the fact that $T$ is a Rota-Baxter operator on the differential  algebra $(P,\cdot,\partial)$, $(P,\ast,\partial)$ is a differential Zinbiel algebra. By the fact that $T$ is a Rota-Baxter operator on the Lie conformal algebra $(P,[\cdot\,_{\lambda}\,\cdot],\partial)$, $(P,\cdot\,_{\lambda}\,\cdot,\partial)$ is a pre-Lie conformal algebra. We only need to show that the compatibility conditions \eqref{eq:pre-PVA1} and \eqref{eq:pre-PVA2} hold. 
	
By the Leibniz rule and the  fact that $T$ is a Rota-Baxter operator, we have
\begin{eqnarray*}
	&&(x_{\lambda}y - y_{-\lambda-\partial}x) * z-x_{\lambda}(y*z)+ y*(x_{\lambda}z)  \\ 
	&=& T\bigl([T(x)_{\lambda}y] - [T(y)_{-\lambda-\partial}x]\bigr) \cdot z-[T(x)_{\lambda}(T(y)\cdot z)] +T(y)\cdot [T(x)_{\lambda}z] \\
	&=& T\bigl([T(x)_{\lambda}y] + [x_{\lambda}T(y)]\bigr) \cdot z -[T(x)_{\lambda}T(y)]\cdot z-T(y)\cdot [T(x)_{\lambda}z]+T(y)\cdot [T(x)_{\lambda}z]\\
	&=& T\bigl([T(x)_{\lambda}y] + [x_{\lambda}T(y)]\bigr) \cdot z -[T(x)_{\lambda}T(y)]\cdot z=0,
\end{eqnarray*}	
	which implies that \eqref{eq:pre-PVA1} holds.
	
	Similarly, we have 
\begin{eqnarray*}
	&&(x* y + y * x)_{-\lambda - \partial} z-  y * (x_{-\lambda - \partial} z) - x * (y_{-\lambda - \partial} z)\\
	&=&[T\bigl(T(x) \cdot y + T(y) \cdot x\bigr)_{-\lambda - \partial} z] -T(y)\cdot[T(x)_{-\lambda - \partial} z] -T(x) \cdot [T(y)_{-\lambda - \partial} z]\\
	&=& [\big(T(x) \cdot T(y)\big)_{-\lambda - \partial} z]  +T(y)\cdot[z_\lambda T(x)]+T(x) \cdot [z_\lambda T(y)] \\
	&=& -[z_\lambda \big(T(x) \cdot T(y)\big)]+ [z_\lambda T(x)] \cdot T(y) +T(x) \cdot [z_\lambda T(y)]=0,
\end{eqnarray*}
which implies that \eqref{eq:pre-PVA2} holds.

Conversely, let $(A,\ast,\cdot\,_{\lambda}\,\cdot,\partial)$ be a pre-Poisson vertex algebra. Then by Proposition \ref{pro:constructPVA}, $P=B\otimes A$ is a Poisson vertex algebra, where $B$ is the $2$-dimension perm algebra with a basis $e,f$ defined by 
$$e\circ e=e,\quad e\circ f=f,\quad f\circ e=f\circ f=0.$$
Then $B\otimes A=(\mathbb{C}e\otimes A)\oplus  (\mathbb{C}f\otimes A)$. Define a linear map $T:P\rightarrow P$ by 
$$T(e\otimes x)=0,\quad T(f\otimes y)= e\otimes y,\quad \forall~x,y\in A.$$
It is clear that $T \tilde{\partial} = \tilde{\partial} T$ and  $T$ is  a Rota–Baxter operator on the differential algebra $(P, \cdot, {\partial})$. We need to verify that $T$ is Rota–Baxter operator on the Lie conformal algebra $(P, [\cdot\,_\lambda\,\cdot],{\partial})$. In fact, Because  $T(e\otimes x)=0$ for any $x\in A$, the condition \eqref{def:baxter-lie-conformal} follows immediately for all cases where $e\otimes x ,e\otimes y\in P$, or $e\otimes x ,f\otimes y\in P$, or $f\otimes y ,e\otimes x\in P$.  For $f\otimes x ,f\otimes y\in P$,  we have  
$$[T(f\otimes x)_{\lambda} T(f\otimes y)] = [(e\otimes x)_{\lambda} (e\otimes y)]=e\otimes (x_\lambda y)+e\otimes (y_{-\lambda-\partial} x),$$
while  
$$[T(f\otimes x)_{\lambda} (f\otimes y)]= f\otimes (x_\lambda y), \quad [(f\otimes x)_{\lambda} T(f\otimes y)] = f\otimes (y_{-\lambda - \partial} x).$$
Hence,  $T$ is a Rota–Baxter operator on the Lie conformal algebra $(P, [\cdot\,_\lambda\,\cdot],{\partial})$. Consequently,  $T$ is  a Rota–Baxter operator on the Poisson vertex algebra $P$. Moreover, the pre-Poisson vertex algebra induced by the Rota–Baxter operator $T$ is given by
\begin{align*}
	(e\otimes x) \ast	(e\otimes y)&=(e\otimes x) \ast	(f\otimes y)=0,&(f\otimes y) \ast	(e\otimes x)&=(e\otimes y) \cdot	(e\otimes x),\\
		(f\otimes x) \ast	(f\otimes y)&=f\otimes (x\ast y),&(e\otimes x)_\lambda	(e\otimes y)&=(e\otimes x)_\lambda	(f\otimes y)=0,\\
	(f\otimes y)_\lambda(e\otimes x)&=[(e\otimes y)_\lambda(e\otimes x)],&	(f\otimes x)_\lambda(f\otimes y)&=f\otimes (x_\lambda y). 
\end{align*}

Define an embedding $j:A\rightarrow P$ by $j(x)=f\otimes x$ for all $x\in A$. It is clear that 
\begin{align*}
	j(x\ast y)&=f\otimes (x\ast y)=	(f\otimes x) \ast	(f\otimes y)=j(x)\ast j(y),\\
	j(x _\lambda y)&=f\otimes (x_\lambda y)=(f\otimes x)_\lambda(f\otimes y)=j(x)_\lambda j(y).
\end{align*}
Thus $A$ can be identified as  a subalgebra of the pre-Poisson vertex algebra $(P,\ast,\cdot\,_\lambda\,\cdot,\partial)$ induced by the Rota–Baxter operator $T$. 
\end{proof}

\section{Pre-Poisson vertex algebra, pre-vertex algebras and deformation quantizations}\label{sec:def}
In this section, we first introduce the notion of a filtration of pre-vertex algebra and show that such a filtration naturally gives rise to a pre-Poisson vertex algebra. We then introduce the concept of a pre-vertex formal deformation of a differential Zinbiel algebra and demonstrate that its classical limit is a pre-Poisson vertex algebra.
\subsection{Pre-Poisson vertex algebras from filtration of pre-vertex algebras}\label{sec:T}
In the following, we generalize the relationship between filtrations of vertex algebras and Poisson vertex algebras to the case of pre-vertex algebras and pre-Poisson vertex algebras.

A {\bf filtration of pre-vertex algebra} $(A,\triangleright ,\triangleleft,\cdot\,_{\lambda}\,\cdot,\partial)$ is an increasing sequence of subspaces  \(A_0 \subseteq A_1 \subseteq A_2 \subseteq \cdots\) such that 
\[
A = \bigcup_{n=0}^{\infty} A_n \quad \text{and} \quad A_n \triangleright A_m + A_n \triangleleft A_m \subseteq A_{n+m}, \quad A_{m{\lambda}}A_{n}\subseteq A_{m+n-1}[\lambda],\quad \partial A_n\subset A_{n}.
\]
In this situation, the associated graded space
\[
\operatorname{Gr}(A) = \bigoplus_{n=1}^{\infty} A_{n}/A_{n-1}
\]
inherits natural operations from pre-vertex algebra $A$. For homogeneous elements $\bar{x}=x+A_{n-1}\in A_n/A_{n-1}$ and $\bar{y}=y+A_{m-1}\in A_m/A_{m-1}$, we define
\begin{align*}
\bar{x} \triangleright \bar{y} = x \triangleright y + &A_{n+m-1} \in A_{n+m}/A_{n+m-1}, \quad\bar{x} \triangleleft \bar{y} = x \triangleleft y + A_{n+m-1} \in A_{n+m}/A_{n+m-1},\\
\bar{x}_{\lambda}\bar{y}&=x_{\lambda}y+A_{n+m-2}[\lambda]\in (A_{n+m-1}/A_{n+m-2})[\lambda],\quad\partial \bar{x} = \overline{\partial x}.
\end{align*}
These operations are well-defined followed by the filtration conditions. 

\begin{pro}
Let  $A= \bigcup_{n=0}^{\infty} A_n$ be a  filtration of a pre-vertex algebra $(A,\triangleright ,\triangleleft,\cdot\,_{\lambda}\,\cdot,\partial)$.  Then  $({\rm Gr}(A),\ast,\cdot\,_{\lambda}\,\cdot,{\partial})$ is a pre-Poisson vertex algebra, where $\ast:A_n/A_{n-1}\times A_m/A_{m-1}\to A_{m+n}/A_{m+n-1}$ is defined by 
\begin{equation}
	\bar{x} \ast\bar{y}=\bar{x} \triangleright \bar{y}=\bar{y} \triangleleft \bar{x},\quad \forall \bar{x}\in A_n/A_{n-1},\bar{y}\in A_m/A_{m-1}.
\end{equation}
\end{pro}
\begin{proof}
Since  $x \triangleright  y-y\triangleleft x\in A_{m+n}$  and $x_\lambda y\in A_{m+n-1}[\lambda]$ for $x\in A_n,y\in A_m$,	by \eqref{eq:pre-VA1} and $\partial A_{m+n-1} \subset A_{m+n-1}$,  we have $\bar{x} \triangleright \bar{y}=\bar{y} \triangleleft \bar{x}$.  Then the conditions \eqref{eq:L-den1} and \eqref{eq:L-den2} in the definition of a L-dendriform  algebra imply that $({\rm Gr}(A),\ast)$ is a Zinbiel algebra. Also, we have
	$${\partial}(\bar{x}\ast \bar{y})=\overline{\partial (x\triangleright y)}=\overline{\partial (x)\triangleright y+ x\triangleright \partial(y)}=\partial(\bar{x})\ast \bar{y}+\bar{x}\ast \partial(\bar{y}).$$ 
	 Thus $({\rm Gr}(A),\ast,\partial)$ is a differential Zinbiel algebra. 
	
	Letting $\bar{x}\in A_n/A_{n-1}$, $\bar{y}\in A_m/A_{m-1}$ and $\bar{z}\in A_p/A_{p-1}$, by the fact that  $A$ is a pre-Lie  conformal algebra, we have
$$
	(\bar{x}_\lambda\bar{y})_{\lambda+\mu}\bar{z} - \bar{x}_\lambda(\bar{y}_\mu\bar{z})
	= (\bar{y}_\mu\bar{x})_{\lambda+\mu}\bar{z} - \bar{y}_\mu(\bar{x}_\lambda\bar{z}),
	$$
which implies that  \(({\rm Gr}(A),\cdot\,_\lambda\,\cdot)\) is a  pre-Lie conformal algebra.

By \eqref{eq:pre-VA3}, for $x\in A_n\), \(y\in A_m\), \(z\in A_p$, we have
\[
x_\lambda(y\triangleright z) = (x_\lambda y - y_{-\lambda-\partial}x)\triangleright z
+ y\triangleright(x_\lambda z) + \int_0^\lambda (x_\lambda y - y_{-\lambda-\partial}x)_\mu z \, d\mu.
\]
Note that $x_\lambda(y\triangleright z), (x_\lambda y - y_{-\lambda-\partial}x)\triangleright z,  y\triangleright(x_\lambda z)$ are all in $A_{n+m+p-1}[\lambda]$, while $\int_0^\lambda (x_\lambda y - y_{-\lambda-\partial}x)_\mu z \, d\mu$ in $A_{n+m+p-2}[\lambda]$. Furthermore, by the fact that $\bar{x} \ast\bar{y}=\bar{x} \triangleright \bar{y}$, we have 
\[
\bar{x}_\lambda(\bar{y}\triangleright\bar{z})
= (\bar{x}_\lambda\bar{y} - \bar{y}_{-\lambda-\partial}\bar{x})\triangleright\bar{z}
+ \bar{y}\triangleright(\bar{x}_\lambda\bar{z}).
\]
	
Similarly, we can get 	
$$
(x * y + y * x)_{-\lambda-\partial} z = y * (x_{-\lambda-\partial} z) + x * (y_{-\lambda-\partial} z).
$$
Thus $({\rm Gr}(A),\ast,\cdot_{\lambda}\cdot,{\partial})$ is a pre-Poisson vertex algebra.
\end{proof}

Recall that a {\bf filtration of non-unital vertex algebra} $(A,\circ,[\cdot\,_{\lambda}\,\cdot],\partial)$ is an increasing sequence of subspaces  \(A_0 \subseteq A_1 \subseteq A_2 \subseteq \cdots\) such that 
\[
A = \bigcup_{n=0}^{\infty} A_n \quad \text{and} \quad A_n \circ A_m  \subseteq A_{n+m}, \quad [A_{m{\lambda}}A_{n}]\subseteq A_{m+n-1}[\lambda],\quad \partial A_n\subset A_{n}.
\]
Furthermore, the associated graded space
\[
\operatorname{Gr}(A) = \bigoplus_{n=1}^{\infty} A_{n}/A_{n-1}
\]
inherits natural operations from the non-unital vertex algebra $A$. For homogeneous elements $\bar{x}=x+A_{n-1}\in A_n/A_{n-1}$ and $\bar{y}=y+A_{m-1}\in A_m/A_{m-1}$, we define
\begin{align*}
	\bar{x} \circ \bar{y} &= x \circ y + A_{n+m-1} \in A_{n+m}/A_{n+m-1}, \\
	[\bar{x}_{\lambda}\bar{y}]&=[x_{\lambda}y]+A_{n+m-2}[\lambda]\in (A_{n+m-1}/A_{n+m-2})[\lambda],\quad\partial \bar{x} = \overline{\partial x}.
\end{align*}
Then $\operatorname{Gr}(A) $ equipped with these structures is a Poisson vertex algebra \cite{LHS}. 

\begin{defi}
Let  $A = \bigcup_{n=0}^{\infty} A_n$ be a filtration of non-unital vertex algebra $(A,\circ,[\cdot\,_{\lambda}\,\cdot],\partial)$ and  let $T:A\rightarrow A$ be a Rota-Baxter operator on the  vertex algebra $A$. We say that $T$ is {\bf compatible} with the filtration if $T(A_n)\subset A_n$ for all $n\in \mathbb{Z}_+$. 
\end{defi}	

It is clear that a filtration of a pre-vertex algebra gives a filtration of the subadjacent non-unital vertex algebra. Conversely, we have
\begin{pro}
	Let $T: A \to A$ be a Rota-Baxter operator on the non-unital vertex algebra $A$ that is compatible with the filtration $A = \bigcup_{n=0}^{\infty} A_n$. Then
	\begin{itemize}
		\item[\rm(a)] the filtration $A = \bigcup_{n=0}^{\infty} A_n$ is a filtration of the pre-vertex algebra $(A,\triangleright ,\triangleleft,\cdot\,_{\lambda}\,\cdot,\partial)$ induced by $T$;
		\item[\rm(b)] $T$ induces a Rota-Baxter operator (also denoted by $T$) on the Poisson vertex algebra $\operatorname{Gr}(A)$;
		\item[\rm(c)] the pre-Poisson vertex algebra structure on $\operatorname{Gr}(A)$ obtained from the filtration of the pre-vertex algebra $(A,\triangleright ,\triangleleft,\cdot\,_{\lambda}\,\cdot,\partial)$ coincides with the pre-Poisson vertex algebra structure induced by the Rota-Baxter operator $T$ on the Poisson vertex algebra $\operatorname{Gr}(A)$.
	\end{itemize}
\end{pro}
\begin{proof}
	It follows by a direct calculation.
\end{proof}	

\begin{ex}
Let  $A= \bigcup_{n=0}^{\infty} A_n$ be a  filtration of a pre-vertex algebra $(A,\triangleright ,\triangleleft,\cdot\,_{\lambda}\,\cdot,\partial)$ and let $({\rm Gr}(A),\ast,\cdot\,_{\lambda}\,\cdot,{\partial})$ be its corresponding pre-Poisson vertex algebra.  By Proposition \ref{pro:RB-embed}, there exists a non-unital vertex algebra $\huaA=A\oplus A'$, where $A'$ is the copy of $A$,  and it also has a filtration $\huaA= \bigcup_{n=0}^{\infty} (A_n\oplus A'_n)$. Then the Rota-Baxter $T:\huaA\rightarrow \huaA$ given by 
$$T|_{A} = 0, \quad T(x') = x, \quad \forall~x \in A$$ 
is compatible with the filtration filtration $\huaA= \bigcup_{n=0}^{\infty} (A_n\oplus A'_n)$. Hence, $\operatorname{Gr}(\huaA)$ is a Poisson vertex algebra with the induced Rota-Baxter operator $T$. Consequently, the pre-Poisson vertex algebra $({\rm Gr}(A),\ast,\cdot\,_{\lambda}\,\cdot,{\partial})$ is a subalgebra of the pre-Poisson vertex algebra obtained from $\operatorname{Gr}(\huaA)$. 
\end{ex}

\subsection{Pre-Poisson vertex algebras and deformation quantizations via Zinbiel algebras}

In this section, we introduce the notion of a pre-vertex formal deformation of a differencial Zinbiel algebra and show that pre-Poisson vertex algebras are the corresponding classical limits.

Let $A$ be a pre-Lie conformal algebra. Then $A[[\hbar]]$ is the completion of the $\mathbb{\Comp}[\hbar] $-module with respect to the $\hbar$-adic topology. It is straightforward to extend the pre-Lie conformal algebra $\lambda$-operation $\circ_\lambda$ to a $ ~\mathbb{\Comp}[[\hbar]] $-bilinear  $\lambda$-operation  $ \circ^\hbar_{\lambda}$ on $ A[[\hbar]] $ by
\begin{equation}
	(\sum_{i\geq 0}x_i\hbar^i)\circ^\hbar_{\lambda}	(\sum_{j\geq 0}y_j\hbar^j)=\sum_{i,j\geq 0}(x_i\circ_\lambda y_j)\hbar^{i+j}\in A[[\hbar]] [[\lambda]].
\end{equation}
Since $	(\sum_{i\geq 0}x_i\hbar^i)\circ^\hbar_{\lambda}	(\sum_{j\geq 0}y_j\hbar^j)$ generally involves  infinite number of powers of $\lambda$, $(A[[\hbar]] ,\circ^\hbar_{\lambda}	)$ is not a pre-Lie conformal algebra. However,  for any $n\in \Nat$, the quotient space $A[[\hbar]] /\hbar^n A[[\hbar]]$ can induce a pre-Lie conformal algebra naturally. It is natural to give the following notion of $\hbar$-adic pre-Lie conformal algebra.
\begin{defi}
	An {\bf $\hbar$-adic pre-Lie conformal algebra}  is a $ \mathbb{\Comp}[[\hbar]] $-module $A[[\hbar]] $, where $A$
	is a vector space, equipped with the $\hbar$-adic topology, a $ \mathbb{C}[\partial] $-module structure and a continuous $ \mathbb{\Comp}[[\hbar]] $-bilinear $\lambda$-operation  $ \circ^\hbar_{\lambda}:A[[\hbar]]\times A[[\hbar]] \rightarrow  A[[\hbar]] [[\lambda]]$  such that for any $n\in \Nat$, $(A[[\hbar]] /\hbar^n A[[\hbar]], \bar{\circ}^\hbar_{\lambda})$ is a  pre-Lie conformal algebra over $ \mathbb{\Comp}[[\hbar]] $, where $\bar{\circ}^\hbar_{\lambda}$ is given by
	$$(x+ \hbar^n A[[\hbar]])\bar{\circ}^\hbar_{\lambda}	(y+\hbar^n A[[\hbar]])=x\circ^\hbar_{\lambda} y+\hbar^n A[[\hbar]],\quad \forall~x,y\in A[[\hbar]].$$
\end{defi}	
\begin{rmk}
	Let $(A[[\hbar]] , \circ^\hbar_{\lambda})$ be an $\hbar$-adic pre-Lie conformal algebra. For any pair of nonnegative integers $i$ and $j$ with $i\leq j$, the natural map $\pi_{ij}$ from $A[[\hbar]] /\hbar^j A[[\hbar]]$ onto $A[[\hbar]] /\hbar^i A[[\hbar]]$ is a pre-Lie  conformal algebra homomorphism. The pre-Lie  conformal algebras $A[[\hbar]] /\hbar^n A[[\hbar]]$ together with these pre-Lie conformal algebra homomorphisms form an inverse system of pre-Lie conformal algebras over $ \mathbb{\Comp}[[\hbar]] $. Then the $\hbar$-adic pre-Lie conformal algebra $(A[[\hbar]] , \circ^\hbar_{\lambda})$  can be seen as an inverse limit of the pre-Lie conformal algebras $A[[\hbar]] /\hbar^n A[[\hbar]]$ over $ \mathbb{\Comp}[[\hbar]] $.
\end{rmk}	

\begin{defi}
 An {\bf $\hbar$-adic non-unital pre-vertex algebra} is quintuple $(A[[\hbar]],\triangleright_\hbar,\triangleleft_\hbar,\circ_{\lambda}^{\hbar},\partial)$, where  $(A[[\hbar]],\circ_{\lambda}^{\hbar},\partial)$ is an $\hbar$-adic pre-Lie conformal algebra and $(A[[\hbar]],\triangleright_\hbar,\triangleleft_\hbar,\partial)$ is a differential L-dendriform algebra such that  for any $n\in \Nat$, $(A[[\hbar]] /\hbar^n A[[\hbar]], \bar{\triangleright}_\hbar,\bar{\triangleleft}_\hbar,\bar{\circ}^\hbar_{\lambda},\partial)$ is a non-unital pre-vertex algebra over $ \mathbb{\Comp}[[\hbar]]$, where $\bar{\triangleright}_\hbar,\bar{\triangleleft}_\hbar$ are given by 
	$$(x+ \hbar^n A[[\hbar]])\bar{\triangleright}^\hbar_{\lambda}	(y+\hbar^n A[[\hbar]])=x\triangleright^\hbar_{\lambda} y+\hbar^n A[[\hbar]],\, (x+ \hbar^n A[[\hbar]])\bar{\triangleleft}^\hbar_{\lambda}	(y+\hbar^n A[[\hbar]])=x\triangleleft^\hbar_{\lambda} y+\hbar^n A[[\hbar]].$$
\end{defi}

\begin{defi}
	Let $(A,\ast,\partial)$ be a differential Zinbiel algebra. A {\bf pre-vertex  formal deformation} of $ A $ consists of 
	 continuous $ \mathbb{\Comp}[[\hbar]] $-bilinear operations $ \triangleright_\hbar,\triangleleft_\hbar:A[[\hbar]]\times A[[\hbar]] \rightarrow  A[[\hbar]] $ and  $\lambda$-operation  $ \circ^\hbar_{\lambda}:A[[\hbar]]\times A[[\hbar]] \rightarrow  A[[\hbar]] [[\lambda]]$  such that $(A[[\hbar]],\triangleright_\hbar,\triangleleft_\hbar,\circ_{\lambda}^{\hbar},\partial)$ carries the structure of an $\hbar$-adic non-unital  pre-vertex algebra over $ \mathbb{\Comp}[[\hbar]] $ satisfying
	$$x \triangleright_\hbar y  = y \triangleleft_\hbar x =x\ast y~\pmod{\hbar},\qquad 
	x\circ_\lambda^\hbar y =0 \pmod{\hbar}.
	$$
\end{defi}

Let $(A[[\hbar]],\triangleright_\hbar,\triangleleft_\hbar,\circ_{\lambda}^{\hbar},\partial)$ be an $\hbar$-adic non-unital pre-vertex algebra. We set  
\begin{eqnarray*}
x\triangleright_\hbar y&=&x\ast y+(x\triangleright_1y)\hbar+(x\triangleright_2y)\hbar^2+\cdots,\\
x\triangleleft_\hbar y&=&y\ast x+(x\triangleleft_1y)\hbar+(x\triangleleft_2y)\hbar^2+\cdots,\\
x\circ_{\lambda}^{\hbar} y &=& \sum_{n=1}^{\infty} \hbar^n m_{n\lambda}(x,y),
\end{eqnarray*}
where $x\triangleright_ny, x\triangleleft_ny:A\times A\to A$ and  $m_{n\lambda}: A \times A \to A[\lambda]$ for \(n \ge 1\) are \(\mathbb{C}\)-bilinear maps.

\begin{thm}
	Let $(A[[\hbar]],\triangleright_\hbar,\triangleleft_\hbar,\circ_{\lambda}^{\hbar},\partial)$ be a pre-vertex formal deformation of a  differential Zinbiel algebra $(A,\ast,\partial)$. Then $(A,\ast,\cdot\,_\lambda\,\cdot,\partial)$ is a pre-Poisson vertex algebra, where the $\lambda$-operation $\cdot\,_\lambda\,\cdot:A\times A\to A[\lambda]$ is defined by 
	$$x_\lambda y=m_{1\lambda}(x,y),\quad \forall x,y\in A,$$ 
	which is called {\bf the classical limit} of $(A[[\hbar]],\triangleright_\hbar,\triangleleft_\hbar,\circ_{\lambda}^{\hbar},\partial)$. The $\hbar$-adic non-unital pre-vertex algebra $ A[[\hbar]]$ is called a {\bf deformation quantization} of $(A,\ast,\cdot\,_\lambda\,\cdot,\partial)$.
\end{thm}

\begin{proof}
	Since \((A[[\hbar]], \circ_\lambda^\hbar\cdot,\partial )\) is an \(\hbar\)-adic pre-Lie conformal algebra, we have
	\[
	(x \circ_\lambda^\hbar y) \circ_{\lambda+\mu}^\hbar z - x \circ_\lambda^\hbar (y \circ_\mu^\hbar z)
	= (y \circ_\mu^\hbar x) \circ_{\lambda+\mu}^\hbar z - y \circ_\mu^\hbar (x \circ_\lambda^\hbar z).
	\]  
	Expanding both sides in powers of \(\hbar\) and comparing the coefficients of \(\hbar^2\), we get
	\[
	m_{1(\lambda+\mu)}\big(m_{1\lambda}(x,y),z\big)- m_{1\lambda}(x, m_{1\mu}(y,z))
	= m_{1(\lambda+\mu)}\big(m_{1\mu}(y,x),z\big)- m_{1\mu}(y, m_{1\lambda}(x,z)).
	\]  
	Also, by the fact that $x\circ_\lambda^\hbar (\partial y)=(\partial+\lambda)x\circ_\lambda^\hbar y$ and $(\partial x)\circ_\lambda^\hbar y=-\lambda x\circ_\lambda^\hbar y$, we have 
$$	m_{1\lambda}(x,\partial y)=(\partial+\lambda)m_{1\lambda}(x,y),\quad m_{1\lambda}(\partial x,y)=-\lambda m_{1\lambda}(x,y).$$
Hence  $(A,\cdot\,_\lambda\,\cdot=m_{1\lambda},\partial)$ is a pre-Lie conformal algebra.
	
	By \eqref{eq:pre-VA2}, we have
	\[
	(y \triangleright_\hbar z + y \triangleleft_\hbar z)_{-\lambda-\partial}x
	= (y_{-\lambda-\partial}x) \triangleleft_\hbar z + y \triangleright_\hbar (z_{-\lambda-\partial}x)
	- \int_0^\lambda z_{-\mu-\partial}(y_{-\lambda-\partial}x) \, d\mu .
	\]  
	The left-hand side becomes  
	\[
	\hbar\, m_{1(-\lambda-\partial)}(y*z+z*y, x) + O(\hbar^2),
	\]  
	while the right-hand side becomes
	$$\hbar\Big(z*m_{1(-\lambda-\partial)l}(y,x)  + y * m_{1(-\lambda-\partial)}(z,x)\Big)+ O(\hbar^2). $$
	Comparing the coefficients of \(\hbar\) on the both sides, we obtain  
	\[
	m_{1(-\lambda-\partial)}(y*z+z*y, x) = z*m_{1(-\lambda-\partial)l}(y,x)  + y * m_{1(-\lambda-\partial)}(z,x).
	\]  
	which implies that condition \eqref{eq:pre-PVA2}  in the definition of pre-Poisson vertex algebra holds  with $x_\lambda y=m_{1\lambda}(x,y)$.

By \eqref{eq:pre-VA3}, we have 
	\[
	x\circ_\lambda (y \triangleleft_\hbar z) = y \triangleleft_\hbar (x\circ_\lambda z - z\circ_{-\lambda-\partial}x) + (x\circ_\lambda y) \triangleleft_\hbar z
	- \int_0^\lambda z\circ_{-\mu-\partial}(x\circ_\lambda y) \, d\mu .
	\]  
	 The left-hand side becomes  
	$$\hbar\, m_{1\lambda}(x, z * y) + O(\hbar^2).$$
	while the right-hand side becomes
	$$\hbar\Big( (m_{1\lambda}(x,z)-m_{1(-\lambda-\partial)}(z,x))*y +z* (x_\lambda y)\Big)+ O(\hbar^2).$$
	Comparing the coefficients of \(\hbar\) on the both sides, we obtain  
$$
	m_{1\lambda}(x, z * y) =  (m_{1\lambda}(x,z)-m_{1(-\lambda-\partial)}(z,x))*y +z* (x_\lambda y) ,
$$
	which implies that  \eqref{eq:pre-PVA1}  in the definition of pre-Poisson vertex algebra holds.
	
Similarly,  the condition \eqref{eq:pre-VA4} in the definition of pre-vertex algebra also implies \eqref{eq:pre-PVA1}  in the definition of pre-Poisson vertex algebra. Thus $(A,\ast,\cdot\,_\lambda\,\cdot,\partial)$ is a pre-Poisson vertex algebra.
\end{proof}

\noindent
{\bf Acknowledgements.} 
We give our warmest thanks to P.S. Kolesnikov for very useful comments and discussions. This research is supported by NSFC (W2412041,12371029) and the National Key Research and Development Program of China (2021YFA1002000).

 \end{document}